\documentclass[11pt]{amsart}
\usepackage{amsmath,amssymb,amsthm,mathtools,mathrsfs}
\usepackage{enumitem}
\usepackage{xcolor}
\usepackage[colorlinks=true,linkcolor=blue,citecolor=blue,urlcolor=blue]{hyperref}
\usepackage{microtype}
\usepackage[a4paper,margin=30mm]{geometry}
\allowdisplaybreaks
\newtheorem{theorem}{Theorem}[section]
\newtheorem{proposition}[theorem]{Proposition}
\newtheorem{lemma}[theorem]{Lemma}

\newtheorem{remark}[theorem]{Remark}
\newtheorem{definition}[theorem]{Definition}
\newtheorem*{definition*}{Definition}
\newtheorem*{remark*}{Remark}
\numberwithin{equation}{section}

\newcommand{\R}{\mathbb R}
\newcommand{\C}{\mathbb C}
\newcommand{\Z}{\mathbb Z}
\newcommand{\Nn}{\mathbb N}
\newcommand{\F}{\mathscr F_\kappa}
\newcommand{\FInv}{\mathscr F_\kappa^{-1}}
\newcommand{\Sscr}{\mathscr S}
\newcommand{\Scirc}{\mathscr S_0}
\newcommand{\Dcal}{\mathcal D}
\newcommand{\M}{\mathcal M}
\newcommand{\MO}{\mathcal M^{\mathcal O}}
\newcommand{\ip}[2]{\langle #1,#2\rangle}
\newcommand{\norm}[1]{\lVert #1\rVert}

\newcommand{\supp}{\operatorname{supp}}
\newcommand{\diam}{\operatorname{diam}}

\newcommand{\dotB}{\dot B}
\newcommand{\dotF}{\dot F}
\newcommand{\bseq}{\dot b}
\newcommand{\fseq}{\dot f}
\newcommand{\Qdim}{\mathbf N}
\newcommand{\1}{\mathbf 1}

\title[A self-dual Dunkl wavelet frame]{A self-dual Dunkl wavelet frame and Besov--Triebel--Lizorkin spaces}
\author{The Anh Bui}
\address{School of Mathematical and Physical Sciences, Macquarie University, NSW 2109, Australia}
\email{the.bui@mq.edu.au}
\date{}

\begin{document}
\begin{abstract}
We construct a self-dual Parseval wavelet frame associated with the Dunkl transform and use it to characterize the corresponding Besov and Triebel–Lizorkin spaces for arbitrary smoothness and the full range of admissible indices. We also establish differential atomic decompositions with atoms supported on dilated Euclidean dyadic cubes, satisfying ordinary derivative estimates and polynomial moment conditions. A key difficulty is the interaction between Euclidean dyadic cubes and the orbit geometry associated with the Dunkl operators. As applications, we obtain fractional chain rules and a Jackson inequality for nonlinear wavelet approximation in \(L^p(dw)\).
\end{abstract}

\maketitle

\section{Introduction}
Wavelets and atomic decompositions play a central role in the theory of
Besov and Triebel--Lizorkin spaces.  In the Euclidean setting, the discrete
transform of Frazier and Jawerth \cite{FrazierJawerth1985,FrazierJawerth}
recovers functions from dyadic coefficients and characterizes their
smoothness in sequence spaces.  Such characterizations are useful in
approximation theory, singular integrals and partial differential equations.
Atomic decompositions provide a complementary description by localized
functions with size, smoothness and cancellation conditions.  Together,
these two tools connect frequency localization with spatial regularity.

The theory has also been developed on spaces of homogeneous type.  Among
the contributions are Littlewood--Paley theory and reproducing formulas of
Han and Sawyer \cite{HanSawyer}, frame characterizations of Yang
\cite{YangFrame}, the treatment of Besov and Triebel--Lizorkin spaces by
Han, M\"uller and Yang \cite{HanMullerYang}, and the regular wavelets of Auscher and Hyt\"onen
\cite{AuscherHytonen}.  In particular, He, Wang, Yang and Yuan
\cite{HeWangYangYuan} obtained wavelet and molecular characterizations of
Besov and Triebel--Lizorkin spaces in a general homogeneous-type setting.
In a related geometric setting, Bownik, Dziedziul and Kamont
\cite{BownikDziedziulKamont} constructed Parseval wavelet frames on
Riemannian manifolds and characterized Besov and Triebel--Lizorkin spaces
on compact manifolds through their frame coefficients.
These developments show the reach of wavelet methods beyond Euclidean
spaces, but do not directly give a self-dual Parseval sampling frame
associated with the Dunkl spectral decomposition.

In the Dunkl setting, spectral multiplier kernels are localized with
respect to the orbits of a finite reflection group. A wavelet-type
decomposition of Dunkl Triebel--Lizorkin spaces was obtained in
\cite{HanLeeLiWick} under restrictions on the smoothness and integrability
indices. More recently, \cite{BuiDunkl} developed the operator-associated
Besov and Triebel--Lizorkin theory, including $L$-atomic decompositions.
The orbit geometry, however, differs from the Euclidean geometry used in
classical wavelet and atomic decompositions.

Our first result is a self-dual Parseval frame obtained from a positive
cubature formula on Euclidean dyadic cubes. It yields wavelet
characterizations of Dunkl Besov and Triebel--Lizorkin spaces for all
$s\in\R$ and the full range of admissible $p,q$, thus extending the
wavelet-type results of \cite{HanLeeLiWick}. Our second result gives a
differential atomic characterization with ordinary derivative estimates,
polynomial moments and support in a single dilated Euclidean cube. This
extends the classical-type atomic results of \cite{BuiDunkl} beyond their
restricted smoothness and integrability ranges. We also obtain fractional
chain rules and a Jackson inequality for nonlinear wavelet approximation.

We recall the setting. Let $R\subset\R^N\setminus\{0\}$ be a normalized reduced root system, $G$ its finite reflection group and $\kappa\geq0$ a $G$-invariant multiplicity function. Fix a positive subsystem $R_+\subset R$. For $\nu\in R$, write $\sigma_\nu$ for the corresponding reflection. The Dunkl measure and its homogeneous dimension are
\[
 dw(x)=\prod_{\nu\in R}|\langle x,\nu\rangle|^{\kappa(\nu)}\,dx,
 \qquad \Qdim=N+\sum_{\nu\in R}\kappa(\nu).
\]
The Dunkl operators are
\[
 T_j f(x)=\partial_j f(x)+\sum_{\nu\in R_+}\kappa(\nu)\nu_j
 \frac{f(x)-f(\sigma_\nu x)}{\langle\nu,x\rangle},
 \qquad 1\leq j\leq N.
\]
They commute, and we write $T^\beta=T_1^{\beta_1}\cdots T_N^{\beta_N}$ and $\Delta_\kappa=\sum_{j=1}^N T_j^2$. If $E$ denotes the Dunkl kernel, the Dunkl transform is
\[
 \F f(\xi)=c_\kappa^{-1}\int_{\R^N}E(-i\xi,x)f(x)\,dw(x).
\]
It is unitary on $L^2(dw)$, and
\[
 \F(T^\beta f)(\xi)=(i\xi)^\beta\F f(\xi),
 \qquad \F(-\Delta_\kappa f)(\xi)=|\xi|^2\F f(\xi).
\]
We put $L=-\Delta_\kappa$.  On the other hand, Dunkl kernel bounds naturally involve the orbit distance
\[
 d(x,y)=\min_{\sigma\in G}|x-\sigma(y)|;
\]
see \cite{AnkerDziubanskiHejna}. The distinction between Euclidean and orbit geometry is important in our construction: our frame is indexed by \emph{Euclidean} cubes, although its elements may be localized near several reflected images of their sampling points.

Let $\Dcal_j$ be the collection of ordinary Euclidean dyadic cubes of sidelength $2^{-j}$ and let $\Dcal=\bigcup_{j\in\Z}\Dcal_j$. For a measurable set $U$, write $w(U)=\int_U dw$; for a cube $Q$, let $\ell(Q)$ denote its sidelength. We use $\1_Q$ for its indicator function. Let $\Scirc'$ denote the dual of $\Scirc$, where
\[
 \Scirc=\{\varphi\in\Sscr(\R^N):
 \partial^\alpha(\F\varphi)(0)=0\text{ for every multi-index }\alpha\}.
\]
We equip $\Scirc$ with the induced Schwartz topology. Fix a real $\Phi\in C_c^\infty((0,\infty))$ satisfying
\begin{equation}\label{eq:phi-support}
 \supp\Phi\subset[1/2,2].
\end{equation}
\begin{equation}\label{eq:square-partition}
 \sum_{j\in\Z}\Phi(2^{-j}t)^2=1\qquad(t>0).
\end{equation}
Set $P_jf=\FInv[\Phi(2^{-j}|\cdot|)\F f]$.

\begin{definition}\label{def:transform-spaces}
Let $s\in\R$.
\begin{enumerate}[label=\textup{(\roman*)}]
\item If $0<p,q\leq\infty$, define $\dotB^{s}_{p,q}(dw)$ to be the set of $f\in\Scirc'$ such that
\begin{equation}\label{eq:B-transform}
 \norm{f}_{\dotB^{s}_{p,q}(dw)}
 :=\Big(\sum_{j\in\Z}
 [2^{js}\norm{P_jf}_{L^p(dw)}]^q\Big)^{1/q}<\infty.
\end{equation}
\item If $0<p<\infty$ and $0<q\leq\infty$, define $\dotF^{s}_{p,q}(dw)$ to be the set of $f\in\Scirc'$ such that
\begin{equation}\label{eq:F-transform}
 \norm{f}_{\dotF^{s}_{p,q}(dw)}
 :=\Big\|
 \Big(\sum_{j\in\Z}[2^{js}|P_jf|]^q\Big)^{1/q}
 \Big\|_{L^p(dw)}<\infty.
\end{equation}
\end{enumerate}
\end{definition}

The usual modifications apply when an exponent is infinite. Since
$\F(Lf)(\xi)=|\xi|^2\F f(\xi)$, we have
$P_j=\Phi(2^{-j}\sqrt L)$. Hence the spaces in
Definition~\ref{def:transform-spaces} agree, with identical quasi-norms
for the same cutoff, with the corresponding operator-associated spaces.
Moreover, they are independent, up to equivalent quasi-norms, of the
choice of admissible annular cutoff $\Phi$. This follows by the standard
finite-overlap Littlewood--Paley argument; see
\cite[Theorem~3.4]{BuiBuiDuongFMS}.
The equivalent coefficient descriptions furnished below by
Theorem~\ref{thm:frame-spaces} are therefore also independent of the
admissible wavelet frame $\Psi$, although the individual coefficients
depend on its choice. When $\kappa=0$, these spaces reduce to the classical
homogeneous Besov and Triebel--Lizorkin spaces.

The difference between orbit and Euclidean geometry will be visible both
in the construction of the frame and in the atomic decomposition.  The
support property for Dunkl convolution makes it possible to apply the
cubature formula to products of multiplier kernels.  This yields the exact
Parseval identity without constructing a separate dual family.

Throughout the paper, $C,c>0$ denote constants that may change from line to line. We write $A\lesssim B$ if $A\leq CB$, and $A\simeq B$ if both $A\lesssim B$ and $B\lesssim A$.

\medskip
\noindent\textbf{Main results.}\quad
For $j\in\Z$, let $K_j(x,y)$ be the kernel of the Dunkl multiplier
\[
 f\longmapsto \FInv\!\Big[\Phi(2^{-j}|\cdot|)\F f\Big].
\]
Fix a sufficiently large integer $J$. At frequency $2^j$ we sample on
$\Dcal_{j+J}$, whose cubes have sidelength $2^{-j-J}$. The fixed shift $J$
provides the fine sampling needed below. We begin with the self-dual Parseval
frame used in the wavelet characterization.

\begin{theorem}[Self-dual Parseval Dunkl wavelet frame]\label{thm:frame}
Choose arbitrarily one point $x_Q\in Q$ for each $Q\in\Dcal_{j+J}$, $j\in\Z$. Then there are positive numbers $\omega_Q\simeq w(Q)$ such that, with
\begin{equation}\label{eq-psi Q}
 \psi_Q(x)=\omega_Q^{1/2}K_j(x,x_Q),
\end{equation}
the family
\[
 \Psi=\{\psi_Q:j\in\Z,\ Q\in\Dcal_{j+J}\}
\]
is a Parseval frame of $L^2(dw)$. More precisely, for every $f\in L^2(dw)$,
\begin{equation}\label{eq:frame-reconstruction}
 f=\sum_{j\in\Z}\sum_{Q\in\Dcal_{j+J}}
 \langle f,\psi_Q\rangle\psi_Q
 \quad\text{in }L^2(dw),
\end{equation}
\[
 \|f\|_2^2=\sum_{j\in\Z}\sum_{Q\in\Dcal_{j+J}}
 |\langle f,\psi_Q\rangle|^2,
\]
and
\begin{equation}\label{eq:coefficient-sampling}
 \langle f,\psi_Q\rangle
 =\omega_Q^{1/2}P_j f(x_Q).
\end{equation}
\end{theorem}

\begin{remark*}
The points $x_Q$ are completely arbitrary subject only to $x_Q\in Q$. No center, distinguished point, or compatibility condition between different cubes is required. The constants in the sampling and frame estimates are uniform over all such choices. The cubature weights $\omega_Q$ may depend on the selected points, but always satisfy $\omega_Q\simeq w(Q)$ with uniform constants.
\end{remark*}

The frame elements are localized around the $G$-orbit of $x_Q$, but the cubes $Q$ themselves are ordinary Euclidean dyadic cubes. They also have cancellation of every polynomial order. See Proposition \ref{prop:frame-properties}. These two facts are important in passing from the transform decomposition to Euclidean dyadic sequence spaces.

For a sequence $a=\{a_Q\}_{Q\in\Dcal_{j+J},\,j\in\Z}$, set
\[
 \|a\|_{\dot b^s_{p,q}}
 :=\Big\{\sum_{j\in\Z}2^{jsq}
 \Big[\sum_{Q\in\Dcal_{j+J}}
 \big(w(Q)^{1/p-1/2}|a_Q|\big)^p\Big]^{q/p}\Big\}^{1/q},
\]
and
\[
 \|a\|_{\dot f^s_{p,q}}
 :=\Big\|\Big\{\sum_{j\in\Z}
 \Big[2^{js}\sum_{Q\in\Dcal_{j+J}}
 w(Q)^{-1/2}|a_Q|\1_Q\Big]^q
 \Big\}^{1/q}\Big\|_{L^p(dw)},
\]
with the usual modifications for infinite exponents. Define the analysis and
synthesis operators associated with $\Psi$ by
\begin{equation}\label{eq:analysis-synthesis-operators}
 S_\Psi f:=\{\langle f,\psi_Q\rangle\}_Q,
 \qquad
 T_\Psi a:=\sum_{j\in\Z}\sum_{Q\in\Dcal_{j+J}}a_Q\psi_Q.
\end{equation}
Here the coefficients of $f\in\Scirc'$ are defined by duality. The synthesis
sum is initially taken over finite sequences and, for the sequence spaces
below, is understood in $\Scirc'$. We obtain the following wavelet
characterization.

\begin{theorem}[Wavelet characterization]\label{thm:frame-spaces}
Let $s\in\R$.
\begin{enumerate}[label=\textup{(\roman*)}]
\item For $0<p,q\le\infty$, the operators
\[
 S_\Psi:\dot B^s_{p,q}(dw)\longrightarrow\dot b^s_{p,q},
 \qquad
 T_\Psi:\dot b^s_{p,q}\longrightarrow\dot B^s_{p,q}(dw)
\]
are bounded, and $T_\Psi S_\Psi f=f$ in $\Scirc'$ for every
$f\in\dot B^s_{p,q}(dw)$. In particular,
\begin{equation}\label{eq:B-frame-equivalence}
 \|f\|_{\dot B^s_{p,q}(dw)}\simeq
 \|S_\Psi f\|_{\dot b^s_{p,q}}.
\end{equation}
\item For $0<p<\infty$ and $0<q\le\infty$, the operators
\[
 S_\Psi:\dot F^s_{p,q}(dw)\longrightarrow\dot f^s_{p,q},
 \qquad
 T_\Psi:\dot f^s_{p,q}\longrightarrow\dot F^s_{p,q}(dw)
\]
are bounded, and $T_\Psi S_\Psi f=f$ in $\Scirc'$ for every
$f\in\dot F^s_{p,q}(dw)$. Consequently,
\begin{equation}\label{eq:F-frame-equivalence}
 \|f\|_{\dot F^s_{p,q}(dw)}\simeq
 \|S_\Psi f\|_{\dot f^s_{p,q}}.
\end{equation}
\end{enumerate}
\end{theorem}

We next turn to differential atomic decompositions. For a Euclidean
cube $Q$, let $3Q$ denote the concentric cube with three times the
sidelength.

\begin{definition}[Dunkl differential atom]\label{def:dunkl-atom}
Let $Q\in\Dcal_\nu$, let $M\in\mathbb N$, and let $L_0\in\mathbb N_0$. A $(p;M,L_0)$-Dunkl differential atom associated with $Q$ is a function $a\in C_c^{2M}(\R^N)$ such that
\[
 \supp a\subset3Q,
 \qquad
 |\partial^\beta a(x)|\le \ell(Q)^{-|\beta|}w(Q)^{-1/p},
 \quad |\beta|\le2M,
\]
and
\begin{equation}\label{eq:atom-moments}
 \int_{\R^N}x^\gamma a(x)\,dw(x)=0,
 \qquad |\gamma|\le L_0.
\end{equation}
Here and below $w(Q)^{-1/\infty}=1$.

\end{definition}
When $\kappa=0$, these are the classical smooth atoms of
\cite{FrazierJawerth1985,FrazierJawerth}. In the Dunkl setting,
\cite[Theorems~3.15 and~3.16]{BuiDunkl} obtained $L$-atomic
decompositions for arbitrary $s\in\R$, following the operator approach of
\cite{BuiBuiDuongFMS}. The operator-independent atoms in
\cite[Definition~3.23]{BuiDunkl} are already indexed by dyadic cubes
$Q$ associated with the Euclidean metric, but satisfy
$\supp a\subset\mathcal O(6Q)$ and only first-order regularity and
mean-zero conditions. The corresponding decompositions in
\cite[Theorems~3.24 and~3.25]{BuiDunkl} were obtained for
$-1<s<1$, with $p(s,1)<p\leq\infty$, $0<q\leq\infty$
in the Besov case, and $p(s,1)<p<\infty$,
$p(s,1)<q\leq\infty$ in the Triebel--Lizorkin case, subject to
the auxiliary conditions in those theorems, where
\[
 p(s,1):=\max\Big\{\frac{\Qdim}{\Qdim+1},
                  \frac{\Qdim}{\Qdim+1+s}\Big\}.
\]
Theorem~\ref{thm:atomic-characterization} extends the smoothness and
integrability ranges and yields atoms supported in a single dilated
Euclidean cube $3Q$, with ordinary derivatives and polynomial moments
of prescribed order.

The atomic classes are formed from expansions
\[
 f=\sum_{\nu\in\mathbb Z}\sum_{Q\in\Dcal_\nu}s_Qa_Q
 \quad\text{in }\Scirc'.
\]
For a scalar family $s=\{s_Q\}_{Q\in\Dcal}$ we denote the Besov atomic
coefficient quasi-norm by
\begin{equation}\label{eq:intro-atomic-b-sequence}
 \|s\|_{\dot b^{s,\mathrm{at}}_{p,q}}
 :=\Big\{\sum_{\nu\in\mathbb Z}2^{\nu sq}
 \Big(\sum_{Q\in\Dcal_\nu}|s_Q|^p\Big)^{q/p}\Big\}^{1/q},
\end{equation}
and the Triebel--Lizorkin atomic coefficient quasi-norm by
\begin{equation}\label{eq:intro-atomic-f-sequence}
 \|s\|_{\dot f^{s,\mathrm{at}}_{p,q}}
 :=\Big\|\Big\{\sum_{\nu\in\mathbb Z}
 \Big[2^{\nu s}\sum_{Q\in\Dcal_\nu}|s_Q|w(Q)^{-1/p}\1_Q\Big]^q
 \Big\}^{1/q}\Big\|_{L^p(dw)}.
\end{equation}
The usual modifications are made when an exponent is infinite. These are
coefficient quasi-norms attached to a given atomic representation; we do not
define the Besov or Triebel--Lizorkin function-space quasi-norms by taking an
infimum over all atomic representations. Instead, the atomic characterization
below is formulated by two quantitative statements: every admissible atomic
representation controls the transform norm, and every function in the
transform space admits at least one atomic representation whose coefficient
quasi-norm is controlled by the transform norm. The same notation is recalled
in Section~4, where the differential atomic decomposition is proved.

\begin{theorem}[Differential atomic characterization]\label{thm:atomic-characterization}
Let $s\in\R$, $M\in\Nn$ and $L_0\in\mathbb N_0$.
\begin{enumerate}[label=\textup{(\roman*)}]
\item Let $0<p,q\le\infty$. If
\[
 2M>s,
 \qquad
 L_0+1+\Qdim+s>\frac{\Qdim}{p\wedge1},
\]
then the class of all $f\in\Scirc'$ admitting a representation
\[
 f=\sum_{\nu\in\Z}\sum_{Q\in\Dcal_\nu}s_Qa_Q
 \quad\text{in }\Scirc',
 \qquad
 s\in\dot b^{s,\mathrm{at}}_{p,q},
\]
with $(p;M,L_0)$-Dunkl differential atoms $a_Q$, is exactly
$\dot B^s_{p,q}(dw)$. More precisely, every such representation satisfies
\[
 \|f\|_{\dot B^s_{p,q}(dw)}
 \lesssim \|s\|_{\dot b^{s,\mathrm{at}}_{p,q}},
\]
and conversely every $f\in\dot B^s_{p,q}(dw)$ has at least one such representation for which
\[
 \|s\|_{\dot b^{s,\mathrm{at}}_{p,q}}
 \lesssim \|f\|_{\dot B^s_{p,q}(dw)}.
\]
\item Let $0<p<\infty$ and $0<q\le\infty$. If
\[
 2M>s,
 \qquad
 L_0+1+\Qdim+s>\frac{\Qdim}{p\wedge q\wedge1},
\]
then the corresponding atomic class is exactly $\dot F^s_{p,q}(dw)$. Every admissible atomic representation satisfies
\[
 \|f\|_{\dot F^s_{p,q}(dw)}
 \lesssim \|s\|_{\dot f^{s,\mathrm{at}}_{p,q}},
\]
and conversely every $f\in\dot F^s_{p,q}(dw)$ admits one for which
\[
 \|s\|_{\dot f^{s,\mathrm{at}}_{p,q}}
 \lesssim \|f\|_{\dot F^s_{p,q}(dw)}.
\]
\end{enumerate}
\end{theorem}

The proof follows the Calder\'on reproducing-formula method of
\cite[Theorems~4.2 and~4.6]{BuiBuiDuongFMS}, together with a transfer from
orbit-supported operator atoms to differential atoms on Euclidean cubes.
We do not claim that the sufficient smoothness and moment conditions in
Theorem~\ref{thm:atomic-characterization} are optimal.

Finally, we give fractional chain rules as an application. Section~2 collects the preliminary estimates. Section~3 proves the self-dual frame and the wavelet characterization. Section~4 treats atomic decomposition, while Section~5 proves the chain rules and a nonlinear approximation estimate. Inhomogeneous variants are recorded in the appendix.

The usual modifications are understood whenever an exponent is infinite.

\section{Preliminaries}
Let $R\subset\R^N\setminus\{0\}$ be a normalized reduced root system and let $G$ be the finite reflection group generated by the reflections $\sigma_\nu$, $\nu\in R$. Fix a positive subsystem $R_+\subset R$. We fix a nonnegative $G$-invariant multiplicity function $\kappa:R\to[0,\infty)$. The associated measure is
\[
 dw(x)=\prod_{\nu\in R}|\ip{x}{\nu}|^{\kappa(\nu)}\,dx.
\]
For measurable $E\subset\R^N$ we continue to write $w(E)=\int_Edw$. We write
\[
 \Qdim:=N+\sum_{\nu\in R}\kappa(\nu)
\]
for the homogeneous dimension. For Euclidean balls $B(x,r)$ one has
\begin{equation}\label{eq:volume}
 w(B(x,r))\simeq r^N\prod_{\nu\in R}(|\ip{x}{\nu}|+r)^{\kappa(\nu)},
\end{equation}
and hence $dw$ is doubling. In particular,
\[\Big(\frac Rr\Big)^N
 \lesssim \frac{w(B(x,R))}{w(B(x,r))}
 \lesssim \Big(\frac Rr\Big)^{\Qdim},
 \qquad 0<r\leq R.\]
We also use the orbit notation
\[
 \mathcal O(E):=\bigcup_{\sigma\in G}\sigma(E),
 \qquad E\subset\R^N,
\]
and the orbit distance
\[
 d(x,y):=\min_{\sigma\in G}|x-\sigma(y)|.
\]
Thus $B^d(x,r)=\{y:d(x,y)<r\}=\mathcal O(B(x,r))$, and
$w(B^d(x,r))\simeq w(B(x,r))$, with constants depending only on $G$.

Let $T_1,\dots,T_N$ be the Dunkl differential--difference operators associated
with $(R,\kappa)$. For a multi-index $\beta\in\Nn_0^N$ we write
\[
 T^\beta:=T_1^{\beta_1}\cdots T_N^{\beta_N}.
\]
The operators $T_j$ commute and preserve polynomials, lowering ordinary
polynomial degree by one. Moreover,
\[\F(T^\beta f)(\xi)=(i\xi)^\beta\F f(\xi),
 \qquad f\in\Sscr(\R^N).\]
A simple consequence of the differential--difference formula is the support
property
\begin{equation}\label{eq:dunkl-derivative-support}
 \supp T^\beta f\subset\mathcal O(\supp f),
 \qquad |\beta|>0,
\end{equation}
for every compactly supported smooth $f$. Notice that
\eqref{eq:dunkl-derivative-support} enlarges the support only by the finite
reflection group; it does not alter the Euclidean dyadic geometry used below.

Let $E$ be the Dunkl kernel and $c_\kappa>0$ the usual normalization constant. For $f\in L^1(dw)$ the Dunkl transform is
\[
 \F f(\xi)=c_\kappa^{-1}\int_{\R^N}E(-i\xi,x)f(x)\,dw(x).
\]
It extends to a unitary operator on $L^2(dw)$ and to an automorphism of the Schwartz class $\Sscr(\R^N)$; see \cite{deJeu,RoslerLecture}. The inverse is
\[
 \FInv g(x)=c_\kappa^{-1}\int_{\R^N}E(i\xi,x)g(\xi)\,dw(\xi).
\]
The Dunkl translation $\tau_x$ and convolution $*_{\kappa}$ are normalized so that
\[
 \F(f*_{\kappa}g)=\F f\,\F g.
\]
For suitable functions one also has
\[
 \F(fg)=c_\kappa^{-1}(\F f)*_{\kappa}(\F g).
\]
The precise value of the harmless normalization constant will never be relevant.

A basic support fact will be used repeatedly. For $x,y\in\R^N$ there is a compactly supported distribution $\gamma_{x,y}$ such that
\[
 \tau_xh(y)=\ip{h}{\gamma_{x,y}},
\]
and
\begin{equation}
	\supp\gamma_{x,y}\subset
 \{z:\,||x|-|y||\leq |z|\leq |x|+|y|\};
 \end{equation}
see \cite{AmriAnkerSifi,Trimeche}. This is the mechanism behind the product spectral property below.

For $\lambda>0$ let
\[
 \Sigma_\lambda
 :=\{f\in L^2(dw):\supp\F f\subset\overline{B(0,\lambda)}\}.
\]
When an additional $L^p$ condition is relevant we write $\Sigma_\lambda^p=\Sigma_\lambda\cap L^p(dw)$.

The following support inclusion extends the one-dimensional formula in
\cite[(2.14)]{AmriGasmiSifi} to arbitrary dimension.

\begin{lemma}[Support of Dunkl convolution]\label{lem:convolution-support}
Let $0<a<b<c$ and $f,g\in C_c^\infty(\R^N)$ satisfy
$\supp f\subset\overline{B(0,a)}$ and
$\supp g\subset\{x:b\le |x|\le c\}$. Then
\begin{equation}\label{eq:amri-convolution}
 \supp(f*_{\kappa} g)\subset\{x:b-a\le |x|\le c+a\}.
\end{equation}
\end{lemma}
\begin{proof}
For $\xi\in\R^N$ and $|\eta|\le a$, the translation support theorem gives
\[
 \supp\gamma_{\xi,-\eta}
 \subset\{z:||\xi|-|\eta||\le |z|\le |\xi|+|\eta|\}.
\]
If $|\xi|<b-a$ or $|\xi|>c+a$, this set is disjoint from
$\supp g$. Hence $\tau_\xi g(-\eta)=0$ for $\eta\in\supp f$,
and the convolution formula gives $(f*_{\kappa} g)(\xi)=0$.

This completes our proof.
\end{proof}

\begin{proposition}[Product of band-limited functions]\label{prop:product}
Let $\lambda_1,\lambda_2>0$. If $f\in\Sigma_{\lambda_1}$ and
$g\in\Sigma_{\lambda_2}$, then $fg\in L^1(dw)\cap L^2(dw)$ and
\[
 \supp\F(fg)\subset\overline{B(0,\lambda_1+\lambda_2)}.
\]
In particular, $\Sigma_\lambda\cdot\Sigma_\lambda\subset\Sigma_{2\lambda}^1$.
\end{proposition}

\begin{proof}
Set $a=\F f$ and $b=\F g$. By H\"older's inequality,
$fg\in L^1(dw)$.  Moreover, inversion, $|E(i\xi,x)|\le1$,
and Cauchy--Schwarz on the finite-measure frequency ball give
$\|f\|_\infty\lesssim_{\lambda_1}\|f\|_2$, so
$fg\in L^2(dw)$.  The product formula reads
$\F(fg)=c_\kappa^{-1}(a*_{\kappa} b)$. If $|\xi|>\lambda_1+\lambda_2$
and $|\eta|\leq\lambda_1$, then
\[
 \supp\gamma_{\xi,-\eta}
 \subset\{z:\ |z|\geq|\xi|-|\eta|>\lambda_2\}.
\]
Hence $\tau_\xi b(-\eta)=0$, and consequently
$(a*_{\kappa} b)(\xi)=0$. The identities are first justified for smooth
compactly supported transforms and then by approximation.

This completes our proof.
\end{proof}

From now on, for a bounded radial function $m(\xi)=m_0(|\xi|)$, we denote by $T_m$ and $K_m(\cdot,\cdot)$ the Dunkl multiplier associated with $m$ and its kernels which are defined by  
\[
 T_mf:=\FInv(m\F f).
\]
and
\[
 K_m(x,y)=\tau_x(\FInv m)(-y)=\tau_{-y}(\FInv m)(x),
\]
respectively.

For the sampling arguments we need only the standard localization and regularity estimates for smooth radial multipliers. We record them for the later use.

\begin{proposition}\label{prop:kernel}
Let $\vartheta:[0,\infty)\to\mathbb C$ be such that its even extension
$\widetilde\vartheta(u)=\vartheta(|u|)$ belongs to $\mathscr S(\mathbb R)$, and put
\[
 \vartheta_t(\xi)=\vartheta(t|\xi|),\qquad t>0.
\]
In particular, this hypothesis holds for smooth annular cutoffs and for
restrictions of even Schwartz functions; smoothness merely on the half-line
at the origin is not sufficient.
Let $K_{\vartheta_t}$ be the kernel of $T_{\vartheta_t}$. For every $M>0$ there is $C_M$ such that
\begin{equation}\label{eq:kernel-size}
 |K_{\vartheta_t}(x,y)|
 \leq \frac{C_M}{w(B(x,t+d(x,y)))}
 \Big(\frac{t}{t+d(x,y)}\Big)^M
 \Big(1+\frac{|x-y|}{t}\Big)^{-2}.
\end{equation}
More generally, for every pair of multi-indices $\alpha,\beta$ and every
$M>0$,
\begin{equation}\label{eq:kernel-euclidean-derivatives}
 |\partial_x^\alpha\partial_y^\beta K_{\vartheta_t}(x,y)|
 \lesssim_{\alpha,\beta,M}
 t^{-|\alpha|-|\beta|}
 \frac1{w(B(x,t+d(x,y)))}
 \Big(\frac{t}{t+d(x,y)}\Big)^M
 \Big(1+\frac{|x-y|}{t}\Big)^{-2}.
\end{equation}
Moreover, if $|x-x'|\leq t$, then
\begin{equation}\label{eq:kernel-reg}
 |K_{\vartheta_t}(x,y)-K_{\vartheta_t}(x',y)|
 \leq C_M\frac{|x-x'|}{t}
 \big[\mathcal K_{M,t}(x,y)+\mathcal K_{M,t}(x',y)\big],
\end{equation}
where $\mathcal K_{M,t}$ denotes the right-hand side of \eqref{eq:kernel-size} without the constant $C_M$. The analogous estimate holds in the second variable. If $\vartheta$ vanishes at zero, then
\[\int_{\R^N}K_{\vartheta_t}(x,y)\,dw(y)
 =\int_{\R^N}K_{\vartheta_t}(y,x)\,dw(y)=0.\]
Finally, for every $1\leq p\leq\infty$ the operators $T_{\vartheta_t}$ are uniformly bounded on $L^p(dw)$.
\end{proposition}

\begin{proof}
Apply \cite[Lemma~2.5]{BuiDunkl} to the even Schwartz function
$\widetilde\vartheta$. This gives
\eqref{eq:kernel-size}--\eqref{eq:kernel-reg}, including the
arbitrary-order ordinary derivative estimates
\eqref{eq:kernel-euclidean-derivatives}. The cancellation follows because the multiplier vanishes at $0$. The $L^p$ bounds follow directly from  \eqref{eq:kernel-size}.

This completes our proof.
\end{proof}

\bigskip

For $a>0$ and a function $h$ whose Dunkl transform is supported in $B(0,c/t)$, define the Peetre-type maximal function
\[
 h^*_{t,a}(x):=\sup_{y\in\R^N}\frac{|h(y)|}{(1+d(x,y)/t)^a}.
\]
We shall use two Hardy--Littlewood maximal operators. The Euclidean maximal operator is
\[
 \M g(x):=\sup_{r>0}\frac1{w(B(x,r))}
 \int_{B(x,r)}|g(y)|\,dw(y),
\]
and the orbit maximal operator is
\[
 \MO g(x):=\sup_{r>0}\frac1{w(B^d(x,r))}
 \int_{B^d(x,r)}|g(y)|\,dw(y).
\]
Since an orbit ball is a union of at most $|G|$ Euclidean balls and $dw$ is $G$-invariant, the usual scalar and vector-valued maximal inequalities hold for $\MO$ with the same exponent ranges as for $\M$.

We now record the following elementary estimates in \cite{BuiBuiDuongFMS}, which will be used frequently in the sequel: for any $H>\Qdim$, $r>0$, and $f\in L^1_{\rm loc}(dw)$,
\begin{equation}
	\label{eq-elementary estimate}
	\int_{\R^N} \frac{1}{w(B(x\wedge y,r))}\Big(1+\frac{d(x,y)}{r}\Big)^{-H} |f(y)|dw(y)\lesssim \mathcal M^{\mathcal O}f(x),
\end{equation}
 where $w(B(x\wedge y,r))=\min\{w(B(x,r)),w(B(y,r))\}$ and the implicit constant is independent of $y, r$ and $f$. Consequently, for any $H>\Qdim$ and $r>0$,
\begin{equation}
	\label{eq-elementary estimate 2}
	\int_{\R^N} \frac{1}{w(B(x\wedge y,r))}\Big(1+\frac{d(x,y)}{r}\Big)^{-H} dw(y)\lesssim 1.
\end{equation}

\begin{lemma}\label{lem:peetre}
	Let $0<r<\infty$ and $a>\Qdim/r$. There exists $C$ such that, whenever
	$\supp\F h\subset B(0,c/t)$,
	\begin{equation}\label{eq:peetre-max}
		h^*_{t,a}(x)\leq C\,[\MO(|h|^r)(x)]^{1/r}.
	\end{equation}
	Consequently, for $0<p<\infty$, $0<q\leq\infty$ and
	$0<r<\min\{p,q\}$ (with $r<p$ when $q=\infty$),
	\begin{equation}\label{eq:vector-peetre}
		\Big\|\Big(\sum_j|h^*_{2^{-j},a}|^q\Big)^{1/q}\Big\|_{L^p(dw)}
		\lesssim
		\Big\|\Big(\sum_j|h_j|^q\Big)^{1/q}\Big\|_{L^p(dw)}
	\end{equation}
	for every sequence $\{h_j\}$ with
	$\supp\F h_j\subset B(0,c2^j)$.
\end{lemma}

\begin{proof}
	 Choose $\eta\in C_c^\infty([0,\infty))$ such that
	$\eta=1$ on $[0,c]$, and define the radial multiplier
	\[
	\eta_t(\xi):=\eta(t|\xi|),\qquad \xi\in\R^N.
	\]
	Then $h=T_{\eta_t}h$. Hence, by Proposition~\ref{prop:kernel}, for every
	sufficiently large $M$,
	\begin{equation}\label{eq:peetre-kernel}
		\begin{aligned}
		|h(y)|
		&\lesssim
		\int_{\R^N}
		\frac{1}{w(B(y,t))}
		\Big(1+\frac{d(y,z)}{t}\Big)^{-M}
		|h(z)|\,dw(z).
	\end{aligned}
	\end{equation}

	Assume first that $r\geq1$. By \eqref{eq:peetre-kernel} and \eqref{eq-elementary estimate},
	\[
	|h(y)|^r
	\lesssim \MO h(x)^r\lesssim 
	\Big(1+\frac{d(x,y)}{t}\Big)^{\Qdim}
	\MO(|h|^r)(x).
	\]
	Therefore
	\[
	\frac{|h(y)|}{(1+d(x,y)/t)^a}
	\lesssim
	\Big(1+\frac{d(x,y)}{t}\Big)^{\Qdim/r-a}
	[\MO(|h|^r)(x)]^{1/r}.
	\]
	Since $a>\Qdim/r$, taking the supremum over $y$ gives
	\eqref{eq:peetre-max}.
	
For $r\in (0,1)$, we first note that
	\[
	|h(z)|^{1-r}
	\leq
	[h^*_{t,a}(x)]^{1-r}
	\Big(1+\frac{d(x,z)}{t}\Big)^{a(1-r)}.
	\]
	Using this in \eqref{eq:peetre-kernel}, together with
	\[
	1+\frac{d(x,z)}{t}
	\lesssim
	\Big(1+\frac{d(x,y)}{t}\Big)
	\Big(1+\frac{d(y,z)}{t}\Big),
	\]
	and then applying \eqref{eq-elementary estimate} with a sufficiently large
 $M$, we obtain
	\[
	|h(y)|
	\lesssim
	[h^*_{t,a}(x)]^{1-r}
	\Big(1+\frac{d(x,y)}{t}\Big)^{a(1-r)+\Qdim}
	\MO(|h|^r)(x).
	\]
	Consequently,
	\[
	\frac{|h(y)|}{(1+d(x,y)/t)^a}
	\lesssim
	[h^*_{t,a}(x)]^{1-r}
	\Big(1+\frac{d(x,y)}{t}\Big)^{\Qdim-ar}
	\MO(|h|^r)(x).
	\]
	Since $ar>\Qdim$, taking the supremum over $y$ yields
	\[
	[h^*_{t,a}(x)]^r\lesssim \MO(|h|^r)(x),
	\]
	which proves \eqref{eq:peetre-max}.
	
	Finally, if $0<r<\min\{p,q\}$, then $p/r>1$ and $q/r>1$. Hence
	\eqref{eq:vector-peetre} follows from \eqref{eq:peetre-max} and the
	Fefferman--Stein vector-valued maximal inequality for $\MO$.

This completes our proof.
\end{proof}
\medskip\noindent\textbf{Sampling and cubature.}\quad
Recall that $\Dcal_\ell$ is the Euclidean dyadic grid of sidelength
$2^{-\ell}$. For $Q\in\Dcal_\ell$ choose an arbitrary point $x_Q\in Q$;
no special choice of this point is required. By \eqref{eq:volume} and doubling,
\begin{equation}\label{eq:cube-ball}
 w(Q)\simeq w(B(x_Q,\ell(Q))).
\end{equation}
The constants are uniform in $Q$ and in the choice of $x_Q\in Q$.

\begin{remark*}
At every scale, the point $x_Q$ may be chosen arbitrarily in each cube $Q$. All constants in the sampling theorem below are uniform over the whole family of choices. In the positive cubature theorem, the weights $\omega_Q$ generally depend on the chosen sampling points, but the comparison $\omega_Q\simeq w(Q)$ is uniform. Consequently the self-dual frame can be constructed from any prescribed family $\{x_Q\in Q\}$.
\end{remark*}

\begin{lemma}\label{lem:oscillation}
Fix $c_0>0$. Let $a>0$ be sufficiently large. If
\[
 \supp\F h\subset B(0,c_0t^{-1})
\]
and $|x-x'|\leq t$, then
\begin{equation}\label{eq:oscillation-point}
 |h(x)-h(x')|
 \lesssim \frac{|x-x'|}{t}\big[h^*_{t,a}(x)+h^*_{t,a}(x')\big].
\end{equation}
\end{lemma}

\begin{proof}
Choose $\eta\in C_c^\infty([0,\infty))$ with $\eta=1$ on $[0,c_0]$.
Recall that $\eta_t(\xi):=\eta(t|\xi|)$. Then $h=T_{\eta_t}h$, and hence
\[
 h(x)-h(x')
 =\int[K_{\eta_t}(x,y)-K_{\eta_t}(x',y)]h(y)\,dw(y).
\]
Using
$|h(y)|\leq h^*_{t,a}(x)(1+d(x,y)/t)^a$, taking the decay order in
Proposition~\ref{prop:kernel} larger than $a+2\Qdim+3$ and applying \eqref{eq-elementary estimate 2}, we obtain
\[
 |h(x)-h(x')|
 \lesssim\frac{|x-x'|}{t}h^*_{t,a}(x).
\]
Interchanging $x$ and $x'$ gives \eqref{eq:oscillation-point}.

This completes our proof.
\end{proof}

\begin{lemma}
\label{lem:euclidean-to-dunkl}
Let $K\in\Nn$ and $r>0$. Suppose $h\in C_c^K(\R^N)$ satisfies
\[
 \|\partial^\alpha h\|_\infty\leq A r^{-|\alpha|},
 \qquad |\alpha|\leq K.
\]
Then, for every multi-index $\beta$ with $|\beta|\leq K$,
\begin{equation}\label{eq:euclidean-to-dunkl}
 \|T^\beta h\|_\infty\leq C_{K,R,\kappa}A r^{-|\beta|}.
\end{equation}
If $\supp h\subset E$, then
\begin{equation}\label{eq:euclidean-to-dunkl-support}
 \supp T^\beta h\subset\mathcal O(E),\qquad |\beta|>0.
\end{equation}
\end{lemma}

\begin{proof}
For a positive root $\alpha$ set
\[
 \partial_\alpha=\alpha\cdot\nabla,
 \qquad
 D_\alpha h(x):=
 \frac{h(x)-h(\sigma_\alpha x)}{\langle\alpha,x\rangle},
\]
where the quotient is understood by its smooth extension across
$\langle\alpha,x\rangle=0$. Since
\[
 x-\sigma_\alpha x
 =2\frac{\langle\alpha,x\rangle}{|\alpha|^2}\alpha,
\]
the fundamental theorem of calculus gives the nonsingular formula
\begin{equation}\label{eq:difference-quotient-integral}
 D_\alpha h(x)
 =\frac{2}{|\alpha|^2}\int_0^1
 \partial_\alpha h\big(\sigma_\alpha x+\theta(x-\sigma_\alpha x)\big)\,d\theta.
\end{equation}
In particular,
$\|D_\alpha h\|_\infty\lesssim\|\nabla h\|_\infty$.
More generally,  for every multi-index $\gamma$,
\begin{equation}\label{eq:difference-quotient-derivatives}
 \|\partial^\gamma D_\alpha h\|_\infty
 \leq C_{\alpha,\gamma}
 \max_{|\eta|\leq |\gamma|+1}\|\partial^\eta h\|_\infty.
\end{equation}

Using the standard formula
\[
 T_jh=\partial_jh+
 \sum_{\alpha\in R_+}\kappa(\alpha)\alpha_jD_\alpha h,
\]
we obtain from \eqref{eq:difference-quotient-derivatives}
\begin{equation}\label{eq:T-Ck-bound}
 \max_{|\gamma|\leq m-1}\|\partial^\gamma T_jh\|_\infty
 \leq C_m\max_{|\eta|\leq m}\|\partial^\eta h\|_\infty,
 \qquad 1\leq m\leq K.
\end{equation}
Iterating \eqref{eq:T-Ck-bound} gives, at unit scale,
\[
 \|T^\beta h\|_\infty
 \leq C_{K,R,\kappa}\max_{|\eta|\leq|\beta|}
 \|\partial^\eta h\|_\infty.
\]
To recover the scale $r$, apply this estimate to $h_r(u):=h(ru)$. The
homogeneity of the Dunkl operators gives
\[
 T^\beta h_r(u)=r^{|\beta|}(T^\beta h)(ru),
\]
whereas $\|\partial^\eta h_r\|_\infty\leq A$ for
$|\eta|\leq K$. This proves \eqref{eq:euclidean-to-dunkl}.

For the support assertion, first observe from the differential--difference
formula that
\[
 \supp T_jh\subset \supp h\cup
 \bigcup_{\alpha\in R_+}\sigma_\alpha(\supp h)
 \subset\mathcal O(E).
\]
Since $\mathcal O(E)$ is $G$-invariant, iteration gives
$\supp T^\beta h\subset\mathcal O(E)$ for every $|\beta|>0$. This is
\eqref{eq:euclidean-to-dunkl-support}.

This completes our proof.
\end{proof}

\section{Self-dual wavelet frame and wavelet characterization}
We prove Theorems~\ref{thm:frame} and~\ref{thm:frame-spaces}. We first establish a sampling inequality and a positive cubature formula on Euclidean dyadic cubes. These yield the Parseval frame, from which the sequence-space characterizations follow.

\begin{theorem}\label{thm:sampling}
Fix $c_0>0$. There exists an integer $J_0\geq1$ with the following property. For every $J\geq J_0$, every $j\in\Z$, and every   band-limited function $h$ satisfying 
\[
 \supp\F h\subset B(0,c_0 2^j),
\]
one has, for all $0<p\leq\infty$,
\begin{equation}\label{eq:sampling-lp}
 \norm{h}_{L^p(dw)}
 \simeq
 \Big(\sum_{Q\in\Dcal_{j+J}}w(Q)|h(x_Q)|^p\Big)^{1/p}.
\end{equation}
The implicit constants are independent of $j$ and $h$.

Moreover, if $0<p<\infty$ and $0<q\leq\infty$, then for every sequence $\{h_j\}_{j\in\Z}$ satisfying $\supp\F h_j\subset B(0,c_0 2^j)$,
\begin{equation}\label{eq:sampling-vector}
 \Big\|\Big(\sum_j |h_j|^q\Big)^{1/q}\Big\|_{L^p(dw)}
 \simeq
 \Big\|\Big(\sum_j \Big|\sum_{Q\in\Dcal_{j+J}}h_j(x_Q)\1_Q\Big|^q\Big)^{1/q}\Big\|_{L^p(dw)}.
\end{equation}
\end{theorem}

\begin{proof}
Set $t=2^{-j}$ and define
\[
 \mathcal S_jh:=\sum_{Q\in\Dcal_{j+J}}h(x_Q)\1_Q.
\]
For $x\in Q\in\Dcal_{j+J}$,
$|x-x_Q|\leq\sqrt N\,2^{-J}t$. Taking $J$ large enough and applying
Lemma~\ref{lem:oscillation},
\[
 |h(x)-h(x_Q)|
 \lesssim2^{-J}[h^*_{t,a}(x)+h^*_{t,a}(x_Q)].
\]
Since $d(x,x_Q)\lesssim2^{-J}t$, the Peetre maximal functions at $x$ and
$x_Q$ are comparable, and therefore
\begin{equation}\label{eq:sampling-point-error}
 |h(x)-\mathcal S_jh(x)|
 \lesssim2^{-J}h^*_{t,a+1}(x).
\end{equation}

Let $0<p<\infty$. Choose $0<r<p$ and $a$ with $a+1>\Qdim/r$. By
Lemma~\ref{lem:peetre},
\begin{equation}\label{eq:error-lp}
 \|h-\mathcal S_jh\|_p\leq C2^{-J}\|h\|_p.
\end{equation}
Put $\rho=\min\{1,p\}$. Since the $L^p$ quasi-norm has a
$\rho$-triangle inequality, \eqref{eq:error-lp} implies, for $J$ sufficiently
large,
\[
 \|\mathcal S_jh\|_p\simeq\|h\|_p.
\]
Because
\[
 \|\mathcal S_jh\|_p^p
 =\sum_{Q\in\Dcal_{j+J}}w(Q)|h(x_Q)|^p,
\]
this proves \eqref{eq:sampling-lp}. The case $p=\infty$ follows directly from
\eqref{eq:sampling-point-error} and the $L^\infty$ Peetre estimate.

For the vector-valued assertion put $E_j=h_j-\mathcal S_jh_j$. Choose
$0<r<\min\{p,q\}$ if $q<\infty$, and $0<r<p$ if $q=\infty$.
By \eqref{eq:sampling-point-error}, Lemma~\ref{lem:peetre}, and the
Fefferman--Stein inequality,
\[\Big\|\Big(\sum_j|E_j|^q\Big)^{1/q}\Big\|_p
 \leq C2^{-J}
 \Big\|\Big(\sum_j|h_j|^q\Big)^{1/q}\Big\|_p.\]
The mixed space $L^p(\ell^q)$ has a $\rho$-triangle inequality with
$\rho=\min\{1,p,q\}$ (and $\rho=\min\{1,p\}$ if $q=\infty$). Applying it to
$h_j=\mathcal S_jh_j+E_j$ and to the reverse identity, and choosing $J_0$ so
that $(C2^{-J_0})^\rho<1/2$, yields \eqref{eq:sampling-vector}.

This completes our proof.
\end{proof}

We next strengthen sampling to an exact integration formula. Positivity of the weights will be important for the self-dual wavelet frame.

\begin{theorem}\label{thm:cubature}
Fix $c_0>0$. There exists $J_1\geq J_0$ such that for every $J\geq J_1$ and every $j\in\Z$ there are numbers
\[
 \omega_Q>0,\qquad Q\in\Dcal_{j+J},
\]
with
\[c\,w(Q)\leq \omega_Q\leq C\,w(Q),\]
such that
\[\int_{\R^N}h(x)\,dw(x)
 =\sum_{Q\in\Dcal_{j+J}}\omega_Qh(x_Q)\]
for every $h\in\Sigma^1_{c_0 2^j}$.
\end{theorem}

\begin{proof}
Set $\mu_Q=w(Q)$. By \eqref{eq:error-lp} with $p=1$,
\begin{equation}\label{eq:quadrature-error}
 \Big|\int h\,dw-\sum_Q\mu_Qh(x_Q)\Big|
 \leq\varepsilon\|h\|_1,
 \qquad h\in\Sigma^1_{c_02^j},
\end{equation}
where $\varepsilon=C2^{-J}$. Applying Theorem \ref{thm:sampling} with $p=1$,
\begin{equation}\label{eq:l1-sampling}
 c_1\|h\|_1\leq\sum_Q\mu_Q|h(x_Q)|\leq c_2\|h\|_1.
\end{equation}

Work first over the real scalars. Let $Y_{\mathbb R}\subset\ell^1(\mu;\mathbb R)$
be the space of sequences $\{h(x_Q)\}_Q$ with real-valued
$h\in\Sigma^1_{c_02^j}$. The lower estimate in \eqref{eq:l1-sampling} makes
the sampling map injective. Hence
\[
 \Lambda(\{h(x_Q)\})
 :=\int h\,dw-\sum_Q\mu_Qh(x_Q)
\]
is well-defined, and \eqref{eq:quadrature-error}--\eqref{eq:l1-sampling} give
\[
 |\Lambda(a)|\leq C\varepsilon\|a\|_{\ell^1(\mu)},\qquad a\in Y_{\mathbb R}.
\]
By Hahn--Banach, $\Lambda$ extends to all of $\ell^1(\mu;\mathbb R)$ with the
same norm. Thus there is a real sequence $\{c_Q\}$ with
$\sup_Q|c_Q|\leq C\varepsilon$ such that
\[
 \Lambda(a)=\sum_Q\mu_Qc_Qa_Q.
\]
It follows that
\[
 \int h\,dw=\sum_Q\mu_Q(1+c_Q)h(x_Q)
\]
for real band-limited $h$.  

For the complex band-limited function  $h\in \Sigma^1_{c_02^j}$, since  the real and imaginary parts of $h$ remain in the same band-limited class.  Applying the real
formula to those two parts gives the identity for complex $h$. Choose $J_1\ge J_0$ so that $C\varepsilon\leq1/2$ for all $J\ge J_1$, and put
$\omega_Q=\mu_Q(1+c_Q)$. Then $\omega_Q>0$ and
$\omega_Q\simeq w(Q)$, proving the theorem.

This completes our proof.
\end{proof}

\medskip\noindent\textbf{Proof of the frame theorem.}\quad
Recall the fixed real cutoff $\Phi$ satisfying \eqref{eq:phi-support}--\eqref{eq:square-partition}. Set
\[
 m_j(\xi):=\Phi(2^{-j}|\xi|),
 \qquad P_jf:=\FInv(m_j\F f),
\]
and let $K_j(x,y)$ be the kernel of $P_j$. Since $m_j$ is real and radial, $P_j$ is self-adjoint on $L^2(dw)$ and
\[
 K_j(x,y)=K_j(y,x).
\]
Moreover, for each fixed $x$,
\[K_j(x,\cdot)\in\Sigma_{2^{j+1}}.\]

Fix once and for all $J\geq J_1$ large enough so that Theorems~\ref{thm:sampling} and~\ref{thm:cubature} apply with bandwidth constant $c_0=4$. For each $j\in\Z$, apply Theorem~\ref{thm:cubature} to the grid $\Dcal_{j+J}$ and denote the resulting cubature weights by $\omega_Q$.

We now prove Theorem~\ref{thm:frame}.

\begin{proof}
Fix $j$. For a fixed $x$, the Dunkl transform in the second variable of
$K_j(x,\cdot)$ is in  $\supp m_j$, so
$K_j(x,\cdot)\in\Sigma_{2^{j+1}}$; similarly
$K_j(\cdot,y)\in\Sigma_{2^{j+1}}$. Proposition~\ref{prop:product} (including its $L^2$ conclusion) gives
\[
 K_j(x,\cdot)K_j(\cdot,y)\in\Sigma^1_{2^{j+2}}.
\]
The cubature formula in Theorem \ref{thm:cubature} and \eqref{eq-psi Q} yield 
\begin{align}
 P_j^2(x,y)
 &=\int K_j(x,z)K_j(z,y)\,dw(z)\notag\\
 &=\sum_{Q\in\Dcal_{j+J}}\omega_QK_j(x,x_Q)K_j(x_Q,y)\notag\\
 &=\sum_{Q\in\Dcal_{j+J}}\psi_Q(x)\psi_Q(y).
 \label{eq:kernel-frame}
\end{align}

By Plancherel theorem,
$\F(P_j^2f)=m_j^2\F f$. Since $\sum_jm_j(\xi)^2=1$ for $\xi\neq0$,
\[
 \sum_jP_j^2f=f
\]
strongly in $L^2(dw)$. To deduce the frame reconstruction from
\eqref{eq:kernel-frame}, we must also justify the infinite sums over the
sampling cubes. We do this below using the Parseval identity.

The kernel is symmetric because $m_j$ is real and radial; hence
\[
 \ip{f}{\psi_Q}=\omega_Q^{1/2}P_jf(x_Q),
\]
which proves \eqref{eq:coefficient-sampling}. Finally,
$P_jf,\overline{P_jf}\in\Sigma_{2^{j+1}}$, so
$|P_jf|^2\in\Sigma^1_{2^{j+2}}$ by Proposition~\ref{prop:product}. Applying Theorem \ref{thm:cubature},
\[
 \sum_Q\omega_Q|P_jf(x_Q)|^2=\|P_jf\|_2^2.
\]
Summing in $j$ and using Plancherel and \eqref{eq:square-partition},
\[
 \sum_{j,Q}|\ip{f}{\psi_Q}|^2
 =\sum_j\|P_jf\|_2^2
 =\int\sum_j|m_j(\xi)|^2|\F f(\xi)|^2\,dw(\xi)
 =\|f\|_2^2.
\]
Recall the analysis and synthesis operators $S_\Psi$ and $T_\Psi$
from \eqref{eq:analysis-synthesis-operators}. The preceding identity gives
\[
 \|S_\Psi f\|_{\ell^2}=\|f\|_{L^2(dw)},
\]
so $S_\Psi:L^2(dw)\to\ell^2$ is an isometry and
$S_\Psi^*S_\Psi=I$. For every finitely supported sequence $c$,
\[
 S_\Psi^*c=\sum_{j,Q}c_Q\psi_Q=T_\Psi c.
\]
Thus $T_\Psi$ extends to $\ell^2$ as the bounded adjoint $S_\Psi^*$.
If $E$ runs through the finite subsets of the index set, then
\[
 \sum_{(j,Q)\in E}\ip{f}{\psi_Q}\psi_Q
 =T_\Psi(\1_E S_\Psi f)
 \longrightarrow S_\Psi^*S_\Psi f=f
 \quad\text{in }L^2(dw).
\]
In particular, the series converges unconditionally in \(L^2(dw)\),
which proves \eqref{eq:frame-reconstruction} and justifies the infinite
cube sums in \eqref{eq:kernel-frame} at the operator level.
Thus the wavelet frame is Parseval and self-dual.

This completes our proof.
\end{proof}

The next proposition summarizes the properties of the wavelet elements needed later.

\begin{proposition}\label{prop:frame-properties}
Let $Q\in\Dcal_{j+J}$. Then:
\begin{enumerate}[label=\textup{(\roman*)}]
\item $\supp\F\psi_Q\subset\{\xi:2^{j-1}\leq|\xi|\leq2^{j+1}\}$.
\item For every multi-index $\alpha$ and every $M>0$,
\begin{equation}\label{eq:frame-derivative-size}
 |\partial^\alpha\psi_Q(x)|
 \lesssim_{\alpha,M}2^{j|\alpha|}w(Q)^{-1/2}
 \Big(1+2^jd(x,x_Q)\Big)^{-M}
 \Big(1+2^j|x-x_Q|\Big)^{-2}.
\end{equation}
In particular, for $\alpha=0$,
\begin{equation}\label{eq:frame-size}
 |\psi_Q(x)|
 \lesssim w(Q)^{-1/2}
 \Big(1+2^jd(x,x_Q)\Big)^{-M}
 \Big(1+2^j|x-x_Q|\Big)^{-2}.
\end{equation}
\item If $|x-x'|\leq c2^{-j}$, then
\begin{equation}\label{eq:frame-reg}
 |\psi_Q(x)-\psi_Q(x')|
 \lesssim 2^j|x-x'|\,w(Q)^{-1/2}
 \sum_{z\in\{x,x'\}}
 \Big(1+2^jd(z,x_Q)\Big)^{-M}.
\end{equation}
\item For every multi-index $\gamma$,
\[\int_{\R^N}x^\gamma\psi_Q(x)\,dw(x)=0.\]
\item For every $0<p\leq\infty$,
\[\norm{\psi_Q}_{L^p(dw)}\lesssim w(Q)^{1/p-1/2}.\]
\end{enumerate}
\end{proposition}

\begin{proof}
Recall that $\psi_Q(x)=\omega_Q^{1/2}K_j(x,x_Q)$.  Item (i) is immediate from the support of $m_j$.
Item (ii) follows directly from  \eqref{eq:kernel-euclidean-derivatives}

Item (iii) is a direct consequence of  the first-variable regularity estimate from
Proposition~\ref{prop:kernel}.

For (iv), the transform $\F\psi_Q$ is a smooth function which vanishes at the origin. The identities
\[
 \F(x^\gamma f)=i^{|\gamma|}T_\xi^\gamma(\F f)
\]
therefore give
\[
 \int_{\R^N}x^\gamma\psi_Q(x)\,dw(x)
 =c_\gamma\,[T_\xi^\gamma(\F\psi_Q)](0)=0.
\]

Item (v) follows directly from \eqref{eq:kernel-euclidean-derivatives} and \eqref{eq-elementary estimate 2}.

This completes our proof.
\end{proof}

\medskip\noindent\textbf{Sequence spaces.}\quad 
For a sequence $a=\{a_Q\}$ indexed by $Q\in\Dcal_{j+J}$, $j\in\Z$, set
\begin{equation}\label{eq:b-sequence}
 \norm{a}_{\bseq^s_{p,q}}
 :=\Big\{
 \sum_{j\in\Z}2^{jsq}
 \Big[\sum_{Q\in\Dcal_{j+J}}
 \big(w(Q)^{1/p-1/2}|a_Q|\big)^p\Big]^{q/p}
 \Big\}^{1/q}
\end{equation}
for $0<p,q\leq\infty$, and
\begin{equation}\label{eq:f-sequence}
 \norm{a}_{\fseq^s_{p,q}}
 :=\Big\|
 \Big\{
 \sum_{j\in\Z}
 \Big[2^{js}
 \sum_{Q\in\Dcal_{j+J}}
 w(Q)^{-1/2}|a_Q|\1_Q
 \Big]^q
 \Big\}^{1/q}
 \Big\|_{L^p(dw)}
\end{equation}
for $0<p<\infty$, $0<q\leq\infty$.

We first record a discrete estimate. Its proof is standard, but the orbit metric in the kernel requires a minor comment.

The next estimates are the Euclidean-dyadic version of
\cite[Lemma~2.3]{BuiDunkl}; that lemma also allows orbit-distance kernel
bounds on ordinary Euclidean cubes. We record both cases together.

\begin{lemma}\label{lem:discrete}
Let $0<r\leq1$ and $M>\Qdim/r$. For every $j\in\Z$ and
nonnegative $\{b_Q\}_{Q\in\Dcal_{j+J}}$,
\[
 \sum_{Q\in\Dcal_{j+J}}b_Q(1+2^jd(x,x_Q))^{-M}
 \lesssim \MO_r\Big(\sum_Qb_Q\1_Q\Big)(x).
\]
Moreover, if $k\geq j$, $0\leq\lambda\leq1$ and $M$ is sufficiently
large, then for nonnegative $\{b_Q\}_{Q\in\Dcal_{k+J}}$,
\begin{align}
 &\sum_{Q\in\Dcal_{k+J}}b_Q
 \Big[\frac{w(Q)}{w(B(x_Q,2^{-j}))}\Big]^\lambda
 (1+2^jd(x,x_Q))^{-M}\notag\\
 &\qquad\lesssim
 2^{\Qdim(k-j)(1/r-\lambda)}
 \MO_r\Big(\sum_{Q\in\Dcal_{k+J}}b_Q\1_Q\Big)(x).
 \label{eq:two-scale-discrete}
\end{align}
The implicit constants may depend on the fixed cube refinement $J$.
\end{lemma}

Both estimates follow directly from \cite[Lemma~2.3]{BuiDunkl}
(take $\lambda=0$, $k=j$ for the first one).

\medskip
We now prove Theorem~\ref{thm:frame-spaces}, using the operators
in \eqref{eq:analysis-synthesis-operators}.

\begin{proof}
We separate analysis, synthesis, and reconstruction.

\emph{Analysis.}
By \eqref{eq:coefficient-sampling} and $\omega_Q\simeq w(Q)$,
\[
 w(Q)^{1/p-1/2}|\ip{f}{\psi_Q}|
 \simeq w(Q)^{1/p}|P_jf(x_Q)|,
 \qquad Q\in\Dcal_{j+J}.
\]
The scalar sampling estimate \eqref{eq:sampling-lp} gives
\[
 \Big(\sum_{Q\in\Dcal_{j+J}}
 [w(Q)^{1/p-1/2}|\ip{f}{\psi_Q}|]^p\Big)^{1/p}
 \simeq\|P_jf\|_{L^p(dw)},
\]
with the usual change for $p=\infty$. Taking the weighted $\ell^q$ norm
proves \eqref{eq:B-frame-equivalence}.

For the Triebel--Lizorkin case, put
\[
 \mathcal S_j h:=\sum_{Q\in\Dcal_{j+J}}h(x_Q)\1_Q.
\]
Since $\Dcal_{j+J}$ partitions $\R^N$, the coefficient identity implies
\[
 \|S_\Psi f\|_{\dot f^s_{p,q}}
 \simeq\left\|\left(\sum_j[2^{js}|\mathcal S_j(P_jf)|]^q
 \right)^{1/q}\right\|_{L^p(dw)}.
\]
The vector-valued sampling estimate \eqref{eq:sampling-vector}, applied to
$h_j=2^{js}P_jf$, then gives
\[
 \|S_\Psi f\|_{\dot f^s_{p,q}}
 \simeq\left\|\left(\sum_j[2^{js}|P_jf|]^q\right)^{1/q}
 \right\|_{L^p(dw)}
 =\|f\|_{\dot F^s_{p,q}(dw)}.
\]
The supremum conventions apply when $q=\infty$.

\emph{Synthesis.}
We first take a finitely supported sequence $a=\{a_Q\}_Q$, and put
\[
 g_j:=\sum_{Q\in\Dcal_{j+J}}w(Q)^{-1/2}|a_Q|\1_Q.
\]
The annular supports imply $P_k\psi_Q=0$ for $|k-j|>3$.
For $|k-j|\leq3$, the kernel estimate gives, for any sufficiently large $H$,
\[
 |P_k\psi_Q(x)|\lesssim_H
 w(Q)^{-1/2}(1+2^jd(x,x_Q))^{-H}.
\]
By Lemma~\ref{lem:discrete}, for $0<r\leq1$ and $H>\Qdim/r$,
\begin{equation}\label{eq:synthesis-point}
 |P_kT_\Psi a(x)|\lesssim\sum_{|j-k|\leq3}\MO_r(g_j)(x).
\end{equation}
Let $0<r<\min\{1,p\}$ if $p<\infty$, or $0<r<1$ if $p=\infty$.
The scalar maximal inequality and finite overlap in $j$ yield
\[
 \|T_\Psi a\|_{\dot B^s_{p,q}(dw)}
 \lesssim\|a\|_{\dot b^s_{p,q}},
\]
where $\|g_j\|_p$ agrees with the corresponding coefficient $\ell^p$
quasi-norm at scale $j$. For the Triebel--Lizorkin spaces, choose
$0<r<\min\{1,p,q\}$ (or $0<r<\min\{1,p\}$ when $q=\infty$).
Applying the vector-valued Fefferman--Stein inequality to
\eqref{eq:synthesis-point} gives
\[
 \|T_\Psi a\|_{\dot F^s_{p,q}(dw)}
 \lesssim\|a\|_{\dot f^s_{p,q}}.
\]

For general sequences, the coefficient bound
$|a_Q|\lesssim 2^{-js}w(Q)^{1/2-1/p}\|a\|$ and the rapid decay of
$\ip{\psi_Q}{\varphi}$ for $\varphi\in\Scirc$ show that the series
in \eqref{eq:analysis-synthesis-operators} defines an element of $\Scirc'$.
The preceding bounds extend by finite truncation and Fatou's lemma.

\emph{Reconstruction.}
Theorem~\ref{thm:frame} gives $T_\Psi S_\Psi=I$ in $L^2(dw)$.
For $f\in\Scirc'$, the rapid scale and spatial decay of the frame
coefficients of each $\varphi\in\Scirc$ allow passage to duality:
\[
 \ip{T_\Psi S_\Psi f}{\varphi}
 =\sum_{j,Q}\ip{f}{\psi_Q}\ip{\psi_Q}{\varphi}
 =\ip{f}{\varphi}.
\]
The last equality follows from the $L^2$ frame identity for test
functions, whose frame expansions converge in $\Scirc$.
Hence $T_\Psi S_\Psi=I$ in $\Scirc'$ and the norm equivalences follow.
The synthesis series also converges in the function-space quasi-norm
when finite sequences are dense; at the remaining endpoints we use
convergence in $\Scirc'$.

This completes our proof.
\end{proof}

\section{Atomic decomposition}
We prove Theorem~\ref{thm:atomic-characterization} using differential and operator atoms.

For a Euclidean cube $Q$, write
\[
3\mathcal O(Q):=\mathcal O(3Q)
=\bigcup_{\sigma\in G}\sigma(3Q).
\]
By Lemma~\ref{lem:euclidean-to-dunkl} and
\eqref{eq:dunkl-derivative-support}, every differential
atom $a$ associated with $Q$ satisfies the corresponding
Dunkl derivative estimates and
\begin{equation}\label{eq:atom-dunkl-support}
	\supp T^\beta a\subset 3\mathcal O(Q),
	\qquad |\beta|\ge 1.
\end{equation}

For a scalar family $s=\{s_Q\}_{Q\in\Dcal}$ define
\begin{equation}\label{eq:atomic-b-sequence}
 \|s\|_{\dot b^{s,\mathrm{at}}_{p,q}}
 :=\Big\{\sum_{\nu\in\Z}2^{\nu sq}
 \Big(\sum_{Q\in\Dcal_\nu}|s_Q|^p\Big)^{q/p}\Big\}^{1/q},
\end{equation}
for $0<p,q\leq\infty$, and
\begin{equation}\label{eq:atomic-f-sequence}
 \|s\|_{\dot f^{s,\mathrm{at}}_{p,q}}
 :=\Big\|\Big\{\sum_{\nu\in\Z}
 \Big[2^{\nu s}\sum_{Q\in\Dcal_\nu}
 |s_Q|w(Q)^{-1/p}\1_Q\Big]^q\Big\}^{1/q}\Big\|_{L^p(dw)},
\end{equation}
for $0<p<\infty$ and $0<q\leq\infty$. The usual changes are made when an
index is infinite.

The two sequence normalizations agree when $c_Q=w(Q)^{1/2-1/p}s_Q$,
up to the fixed factor $2^{Js}$ arising from $\nu=j+J$.

We first estimate $P_ja_Q$ at high and low frequencies.

\begin{lemma}
\label{lem:dunkl-atom-two-scale}
Let $a_Q$ be a $(p;M,L_0)$-Dunkl differential atom associated with
$Q\in\Dcal_\nu$, and put $r=\ell(Q)=2^{-\nu}$.  For every $A>0$ there is a
constant $C_A$, independent of $Q,\nu,j$, such that the following estimates
hold.

If $j\geq\nu$, then
\begin{equation}\label{eq:dunkl-atom-high}
 |P_ja_Q(x)|
 \leq C_A w(Q)^{-1/p}2^{-2M(j-\nu)}
 (1+2^\nu d(x,x_Q))^{-A}.
\end{equation}
If $j<\nu$, then
\begin{equation}\label{eq:dunkl-atom-low}
 |P_ja_Q(x)|
 \leq C_A w(Q)^{-1/p}2^{-(\nu-j)(L_0+1)}
 \frac{w(Q)}{w(B(x_Q,2^{-j}))}
 (1+2^jd(x,x_Q))^{-A}.
\end{equation}

\end{lemma}

\begin{proof}
For $j\geq\nu$, set $t=2^{-j}\leq r$ and
\[
 \widetilde\Phi(\lambda)=\lambda^{-2M}\Phi(\lambda),
 \qquad \lambda>0.
\]
Since $\Phi$ is supported away from the origin,
$\widetilde\Phi\in C_c^\infty((0,\infty))$.  If $\widetilde P_j$ denotes the
Dunkl multiplier with symbol $\widetilde\Phi(2^{-j}|\xi|)$, then
\begin{equation}\label{eq:atom-high-factorization}
 P_ja_Q=t^{2M}\widetilde P_jL^Ma_Q,
 \qquad L=-\sum_{\ell=1}^NT_\ell^2.
\end{equation}
The operator $L^M$ is a finite linear combination of $T^\beta$ with
$|\beta|=2M$.  Definition~\ref{def:dunkl-atom} therefore gives
\begin{equation}\label{eq:LM-atom}
 |L^Ma_Q(y)|\lesssim r^{-2M}w(Q)^{-1/p},
 \qquad
 \supp L^Ma_Q\subset 3\mathcal O(Q).
\end{equation}
For $y\in3\mathcal O(Q)$,
$d(y,x_Q)\lesssim r$ and
$w(3\mathcal O(Q))\lesssim w(Q)$, by $G$-invariance and finiteness of $G$.

For $d(x,x_Q)\lesssim r$, Proposition~\ref{prop:kernel} and
\eqref{eq:atom-high-factorization}--\eqref{eq:LM-atom} give
\[
 |P_ja_Q(x)|\lesssim(t/r)^{2M}w(Q)^{-1/p}.
\]
If $R=d(x,x_Q)\gg r$, then $d(x,y)\simeq R$ for
$y\in 3\mathcal O(Q)$.  Doubling and sufficiently large kernel decay give
\[
 |P_ja_Q(x)|\lesssim(t/r)^{2M}w(Q)^{-1/p}
 (1+R/r)^{-A},
\]
which is \eqref{eq:dunkl-atom-high}.

For $j<\nu$, put $t=2^{-j}>r$. Let
$K_j(x,y)$ be the kernel of $P_j$.  For fixed $x$, let
$\mathcal T^{L_0}_{x,Q}(y)$ be the ordinary Euclidean Taylor polynomial of
$y\mapsto K_j(x,y)$ at $x_Q$ of total degree $L_0$.  Since
$\supp a_Q\subset 3Q$ and $a_Q$ has the moments
\eqref{eq:atom-moments},
\begin{equation}\label{eq:atom-taylor-subtraction}
 P_ja_Q(x)=\int_{3Q}
 [K_j(x,y)-\mathcal T^{L_0}_{x,Q}(y)]a_Q(y)\,dw(y).
\end{equation}
By \eqref{eq:kernel-euclidean-derivatives}, Taylor's formula, and
$|y-x_Q|\lesssim r$ on $3Q$,
\[
 |K_j(x,y)-\mathcal T^{L_0}_{x,Q}(y)|
 \lesssim
 (r/t)^{L_0+1}\frac1{w(B(x_Q,t))}
 (1+d(x,x_Q)/t)^{-A},
\]
after increasing the kernel decay exponent and using doubling to move the
volume from $y$ to $x_Q$.  Since
$|a_Q|\leq w(Q)^{-1/p}$ and $w(3Q)\lesssim w(Q)$,
\eqref{eq:atom-taylor-subtraction} gives \eqref{eq:dunkl-atom-low}.

This completes our proof.
\end{proof}

\begin{proof}[Proof of the first implication in Theorem~\ref{thm:atomic-characterization}]
Let $f=\sum_{\nu,Q}s_Qa_Q$ and put
\[
 H_\nu=\sum_{Q\in\Dcal_\nu}|s_Q|w(Q)^{-1/p}\1_Q.
\]
Choose $0<r_0<p\wedge q\wedge1$ for Triebel--Lizorkin, or
$0<r_0<p\wedge1$ for Besov, close to the upper endpoint.
By Lemmas~\ref{lem:dunkl-atom-two-scale} and~\ref{lem:discrete},
\begin{align}
 2^{js}|P_jf(x)|
 &\lesssim\sum_{\nu\le j}2^{-(j-\nu)(2M-s)}
 2^{\nu s}\MO_{r_0}H_\nu(x)\\
 &\quad+\sum_{\nu>j}2^{-(\nu-j)(L_0+1+\Qdim+s-\Qdim/r_0)}
 2^{\nu s}\MO_{r_0}H_\nu(x).
 \label{eq:atomic-synthesis-direct}
\end{align}
The hypotheses make both scale kernels summable. Discrete convolution
and the maximal inequalities give
\begin{equation}\label{eq:atom-synthesis-B}
 \|f\|_{\dot B^s_{p,q}(dw)}\lesssim\|s\|_{\dot b^{s,\mathrm{at}}_{p,q}},
\end{equation}
\begin{equation}\label{eq:atom-synthesis-F}
 \|f\|_{\dot F^s_{p,q}(dw)}\lesssim\|s\|_{\dot f^{s,\mathrm{at}}_{p,q}}.
\end{equation}
The distributional convergence follows by applying these estimates to
finite partial sums and pairing with test functions in $\Scirc$; see also
\cite{BuiDunkl, FrazierJawerth} for the standard argument.

This completes our proof.
\end{proof}

\bigskip

For the reverse inclusion we first construct smooth $L$-atoms, then
decompose them into differential atoms.

\begin{definition}[Smooth operator atom]\label{def:smooth-L-atom}
Fix $0<p\le\infty$, integers $K\ge1$ and $M\ge1$. A function
$A_Q$ is called a smooth $(L,K,p;2M)$-atom associated with
$Q\in\Dcal_\nu$ if there is $b_Q\in\mathcal D(L^K)$ such that
\begin{enumerate}[label=\textup{(\roman*)}]
\item $A_Q=L^Kb_Q$;
\item for $0\le m\le K$,
\[
 \supp L^mb_Q\subset 3\mathcal O(Q),
 \qquad
 |L^mb_Q(x)|\lesssim \ell(Q)^{2(K-m)}w(Q)^{-1/p};
\]
\item for every $|\alpha|\le2M$,
\begin{equation}\label{eq:smooth-L-atom-derivative}
 |\partial^\alpha A_Q(x)|
 \lesssim \ell(Q)^{-|\alpha|}w(Q)^{-1/p}.
\end{equation}
\end{enumerate}
The implicit constants are independent of $Q$.
\end{definition}

Unlike \cite[Definition~3.14]{BuiDunkl}, we use Euclidean dyadic cubes
and impose \eqref{eq:smooth-L-atom-derivative}.

\begin{lemma}\label{lem:smooth-L-atomic-decomp}
Let $s\in\R$, $M\in\Nn$, $0<p,q\le\infty$, and let $K$ be chosen sufficiently large. The Besov spaces
$\dot B^s_{p,q}(dw)$, and the Triebel--Lizorkin spaces
$\dot F^s_{p,q}(dw)$ when $p<\infty$, admit decompositions into
smooth $(L,K,p;2M)$-atoms with the same coefficient sequence norms as in
\eqref{eq:atomic-b-sequence} and \eqref{eq:atomic-f-sequence}. More precisely,
for $f\in\dot B^s_{p,q}(dw)$ there are coefficients $u_Q$ and smooth
operator atoms $A_Q$ such that
\[
 f=\sum_{\nu\in\Z}\sum_{Q\in\Dcal_\nu}u_QA_Q
 \quad\hbox{in }\Scirc',
 \qquad
 \|u\|_{\dot b^{s,\mathrm{at}}_{p,q}}
 \lesssim \|f\|_{\dot B^s_{p,q}(dw)}.
\]
For $f\in\dot F^s_{p,q}(dw)$ the analogous representation holds with
\begin{equation}\label{eq:L-atomic-expansion-F}
 \|u\|_{\dot f^{s,\mathrm{at}}_{p,q}}
 \lesssim \|f\|_{\dot F^s_{p,q}(dw)}.
\end{equation}
\end{lemma}

\begin{proof}
We follow \cite[Theorems~4.2 and~4.6]{BuiBuiDuongFMS}, using Euclidean cubes.

Choose an even function $\varphi\in C_c^\infty(\mathbb R)$
such that
\[
\supp\varphi\subset(-1,1),
\qquad \int_{\mathbb R}\varphi(u)\,du=2\pi,
\]
and denote by $\Theta$ the Fourier transform of $\varphi$. Then $\Theta$ is an even Schwartz function with
$\Theta(0)=1$. Choose
$\psi\in C_c^\infty((0,\infty))$ supported in an
annulus such that
\[
c_{\psi,\Theta}
:=\int_0^\infty\psi(u)\Theta(u)\,\frac{du}{u}
\ne0,
\qquad
\psi_K(u)=u^{-2K}\psi(u).
\]
The Calder\'on reproducing formula gives
\begin{equation}\label{eq:app-calderon}
	f=c_{\psi,\Theta}^{-1}
	\int_0^\infty
	\psi(t\sqrt L)\Theta(t\sqrt L)f\,\frac{dt}{t}
\end{equation}
in $\Scirc'$.

For $Q\in\Dcal_\nu$, put $\ell(Q)=2^{-\nu}$,
$I_\nu=[2^{-\nu-1},2^{-\nu}]$, and
\[
H_\nu(y)=\int_{I_\nu}
|\psi_K(t\sqrt L)f(y)|\,\frac{dt}{t},
\qquad
u_Q=w(Q)^{1/p}\sup_{y\in Q}H_\nu(y).
\]
When $u_Q\ne0$, set
\begin{equation}\label{eq:app-bQ}
	b_Q=\frac{c_{\psi,\Theta}^{-1}}{u_Q}
	\int_{I_\nu}t^{2K}\Theta(t\sqrt L)
	\big[\1_Q\psi_K(t\sqrt L)f\big]\,\frac{dt}{t},
	\qquad A_Q=L^Kb_Q.
\end{equation}
Omit $u_Q=0$. Since
\[
t^{2K}L^K\psi_K(t\sqrt L)=\psi(t\sqrt L),
\]
formula \eqref{eq:app-calderon} and the
dyadic cube partition yield
\[
f=\sum_{\nu\in\mathbb Z}
\sum_{Q\in\Dcal_\nu}u_QA_Q
\quad\text{in }\Scirc'.
\]

Put $\Theta_m(u)=u^{2m}\Theta(u)$. Since
$\operatorname{supp}\widehat{\Theta_m}\subset(-1,1)$, finite propagation
\cite[Lemma~2.6]{BuiDunkl} gives
$K_{\Theta_m(t\sqrt L)}(x,y)=0$ when $d(x,y)>t$.

Since $t\in I_\nu$ implies $t\le\ell(Q)$,
\[
\supp L^mb_Q\subset 3\mathcal O(Q),
\qquad 0\le m\le K.
\]

The kernel estimates and doubling give
\[
 |L^mb_Q(x)|\lesssim
 \ell(Q)^{2(K-m)}w(Q)^{-1/p},\qquad 0\le m\le K.
\]
Also, Proposition~\ref{prop:kernel} gives, for $|\alpha|\le2M$,
\[
 |\partial_x^\alpha K_{(t^2L)^K\Theta(t\sqrt L)}(x,y)|
 \lesssim t^{-|\alpha|}w(B(y,t))^{-1}
 \1_{\{d(x,y)\le ct\}}.
\]
Applied to \eqref{eq:app-bQ}, this yields
$|\partial^\alpha A_Q(x)|\lesssim
\ell(Q)^{-|\alpha|}w(Q)^{-1/p}$.

Since $d(x,y)\lesssim\ell(Q)$ for $x,y\in Q$, we have
\[
 \sum_{Q\in\Dcal_\nu}|u_Q|w(Q)^{-1/p}\1_Q(x)
 \lesssim H_{\nu,a}^*(x),
 \quad
 H_{\nu,a}^*(x)=\sup_y
 \frac{H_\nu(y)}{(1+2^\nu d(x,y))^a}.
\]
As in \cite[Theorems~4.2 and~4.6]{BuiBuiDuongFMS}, for $0<r<\min\{1,p,q\}$ (or $r<\min\{1,p\}$ when $q=\infty$),
\begin{equation}\label{eq:app-peetre-comparison}
 H_{\nu,a}^*(x)\lesssim
 \sum_{|j-\nu|\le C_0}\MO_r(P_jf)(x),\qquad a>\Qdim/r.
\end{equation}
Maximal inequalities give
\[
 \|u\|_{\dot b^{s,\mathrm{at}}_{p,q}}
 \lesssim\|f\|_{\dot B^s_{p,q}(dw)},\qquad
 \|u\|_{\dot f^{s,\mathrm{at}}_{p,q}}
 \lesssim\|f\|_{\dot F^s_{p,q}(dw)}.
\]
The scale summation follows as in
\cite[Theorems~4.2 and~4.6]{BuiBuiDuongFMS}.

This completes our proof.
\end{proof}

\begin{lemma}
\label{lem:smooth-moment-correction}
Let $\ell_0,K\in\mathbb N_0$ be arbitrary. Write $\mathcal P_{\ell_0}$
for the polynomials on $\R^N$ of  degree at most $\ell_0$.
Let $U=B(x_0,R)$ be a Euclidean ball, and let
$\Lambda$ be a linear functional on $\mathcal P_{\ell_0}$ such that
\begin{equation}\label{eq:moment-functional-bound}
 |\Lambda(P)|\leq A\,w(U)^{-1/2}\|P\|_{L^2(U,dw)},
 \qquad P\in\mathcal P_{\ell_0}.
\end{equation}
Then there exists $q\in C_c^\infty(2U)$ such that
\begin{equation}\label{eq:moment-corrector-identity}
 \int_{\R^N}q(x)P(x)\,dw(x)=\Lambda(P),
 \qquad P\in\mathcal P_{\ell_0},
\end{equation}
and, for $|\alpha|\leq K$,
\begin{equation}\label{eq:smooth-corrector-bound}
 |\partial^\alpha q(x)|
 \leq C_{\ell_0,K}A R^{-|\alpha|}w(U)^{-1},
 \qquad x\in\R^N.
\end{equation}
The constant is independent of $U$, $A$, and $\Lambda$.
\end{lemma}

\begin{proof}
Choose $\eta\in C_c^\infty(B(0,2))$ with $0\leq\eta\leq1$ and
$\eta=1$ on $B(0,1)$, and put
$\eta_U(x)=\eta((x-x_0)/R)$.  Let $d_{\ell_0}=\dim\mathcal P_{\ell_0}$ and choose a real
orthonormal basis $\{p_1,\ldots,p_{d_{\ell_0}}\}$ of $\mathcal P_{\ell_0}$ in
$L^2(U,dw)$.  Consider the matrix
\[
 G_{mn}=\int_{\R^N}\eta_U(x)p_m(x)p_n(x)\,dw(x).
\]
For $P_v=\sum_nv_np_n$ and $v=(v_n)$,
\[
 v^TGv=\int\eta_U|P_v|^2\,dw
 \geq\int_U|P_v|^2\,dw=|v|^2.
\]
Hence $\|G^{-1}\|\leq1$.

We use the uniform polynomial estimate
\begin{equation}\label{eq:orthonormal-polynomial-bound}
 \sup_{2U}|\partial^\alpha P|
 \lesssim_{\ell_0,K}R^{-|\alpha|}w(U)^{-1/2}
 \|P\|_{L^2(U,dw)},\qquad
 P\in\mathcal P_{\ell_0},\quad |\alpha|\le K.
\end{equation}
Indeed, scaling and finite-dimensional norm equivalence give
$\sup_{2U}|\partial^\alpha P|
\lesssim R^{-|\alpha|}\sup_U|P|$.
If $|P(v)|=\sup_U|P|$, the same polynomial norm comparison gives a
fixed $\delta>0$, depending only on $\ell_0,N$, and a ball
$B(v',\delta R)\subset U$ on which $|P|\ge c\sup_U|P|$.
Here $|v-v'|\lesssim\delta R$; a slight inward movement covers the
case $v\in\partial U$. By \eqref{eq:volume} and doubling,
$w(B(v',\delta R))\gtrsim_{\delta}w(U)$,
independently of $x_0$ and $R$. Integrating $|P|^2$ on this ball
proves \eqref{eq:orthonormal-polynomial-bound}.

Set $b_n=\Lambda(p_n)$ and let $\lambda=G^{-1}b$.  From
\eqref{eq:moment-functional-bound}, $|b_n|\leq A w(U)^{-1/2}$, hence the same
bound, up to a fixed constant, holds for every $\lambda_n$.  Define
\[
 q=\eta_U\sum_{n=1}^{d_{\ell_0}}\lambda_np_n.
\]
If $P=\sum_ma_mp_m$, then
\[
 \int qP\,dw=a^TG\lambda=a^Tb=\Lambda(P),
\]
which proves \eqref{eq:moment-corrector-identity}.  The derivative estimate
\eqref{eq:smooth-corrector-bound} follows from Leibniz' rule,
\eqref{eq:orthonormal-polynomial-bound}, and
$|\partial^\beta\eta_U|\lesssim R^{-|\beta|}$.

This completes our proof.
\end{proof}

We now pass from smooth $L$-atoms to differential atoms.

\begin{proposition}
\label{prop:Latom-to-atoms}
Fix \(0<p\leq\infty\), \(M\in\Nn\), and
\(L_0\in\mathbb N_0\).  Let \(E>0\).  Choose the auxiliary order
\(K\) sufficiently large, depending on \(E,M,L_0,p\) and the Dunkl structure.
If \(A_Q\) is a smooth \((L,K,p;2M)\)-atom associated with
\(Q\in\Dcal_\nu\), then
\begin{equation}\label{eq:Latom-differential-decomposition}
 A_Q=\sum_{\ell=0}^{L_Q}\sum_{h=1}^{N_{\ell,Q}}
 \lambda_{\ell,h,Q}a_{\ell,h,Q}
 \qquad\hbox{in }L^1(dw)\hbox{ and in }\Scirc',
\end{equation}
where \(N_{\ell,Q}\leq C_G\), each \(a_{\ell,h,Q}\) is a
\((p;M,L_0)\)-Dunkl differential atom associated with an ordinary Euclidean
dyadic cube \(R_{\ell,h,Q}\), and
\begin{equation}\label{eq:Latom-geometry}
 \ell(R_{\ell,h,Q})\simeq2^\ell\ell(Q),
 \qquad
 d(x_{R_{\ell,h,Q}},x_Q)\lesssim\ell(R_{\ell,h,Q}).
\end{equation}
Moreover,
\begin{equation}\label{eq:Latom-coefficient}
 |\lambda_{\ell,h,Q}|\lesssim2^{-\ell E},
\end{equation}
and
\begin{equation}\label{eq:Latom-weighted-coefficient}
 |\lambda_{\ell,h,Q}|w(R_{\ell,h,Q})^{-1/p}
 \lesssim2^{-\ell E}w(Q)^{-1/p}.
\end{equation}
The integer \(L_Q\) may depend on   \(x_Q\), but
all constants above are independent of \(Q\).
\end{proposition}

\begin{proof}
Put \(r=\ell(Q)\), \(Z=\mathcal O(x_Q)\), and \(J_G=|G|\).  Choose a fixed constant $A_*>10\sqrt N$ (larger if necessary below) and let
\[
 D_\ell=A_*\,2^\ell r,\qquad \ell=0,1,\dots.
\]
For each $\ell$, join $z,z\prime\in Z$ when $|z-z\prime|<D_\ell$.  Let \(\mathscr C_\ell\) be the family of graph-theoretic
connected components.  Thus \(z,z'\) lie in the same component precisely when
there is a finite chain
\[
 z=z_0,z_1,\dots,z_m=z'
\]
with \(|z_i-z_{i+1}|<D_\ell\) for every \(i\).  Each component has at most $J_G$ vertices, so
\begin{equation}\label{eq:Latom-cluster-diameter}
 \diam C\leq(J_G-1)D_\ell,
 \qquad C\in\mathscr C_\ell,
\end{equation}
where \(\diam\) is the Euclidean diameter.  Each component at level \(\ell\)
is contained in a unique component at level \(\ell+1\).  Since \(Z\) is finite,
there is a finite \(L_Q\) for which \(\mathscr C_{L_Q}\) consists of one
component.

Choose the same radial cutoff \(\rho\in C_c^\infty(B(0,1/2))\), with
\(\rho=1\) on \(B(0,1/4)\), and define
\[
 \chi_{\ell,C}(x)=1-\prod_{z\in C}
 \Big[1-\rho\Big(\frac{x-z}{D_\ell}\Big)\Big].
\]
Choose \(x_{\ell,C}\in C\) and put
\(U_{\ell,C}=B(x_{\ell,C},8J_GD_\ell)\).  Then
\begin{equation}\label{eq:Latom-cutoff-bounds}
 \supp\chi_{\ell,C}\subset U_{\ell,C},
 \qquad
 |\partial^\alpha\chi_{\ell,C}|\lesssim_\alpha D_\ell^{-|\alpha|}.
\end{equation}
Cutoffs for distinct components have disjoint supports.  Since
$A_*>10\sqrt N$, every point of $\supp A_Q\subset 3\mathcal O(Q)$ lies in the region where
the cutoff belonging to its level-$0$ component is equal to one.  Thus $\sum_{C\in\mathscr C_0}\chi_{0,C}=1$ on $\supp A_Q$. Put
\[
 A_{\ell,C}:=A_Q\chi_{\ell,C},
\]
we have
\begin{equation}\label{eq:Latom-partition-on-support}
 A_Q=\sum_{C\in\mathscr C_\ell}A_{\ell,C}
 \qquad\text{for every }0\leq\ell\leq L_Q.
\end{equation}

We next estimate the local polynomial moments.  Since \(A_Q=L^Kb_Q\),
self-adjointness gives, for \(P\in\mathcal P_{L_0}\),
\[
 \int A_{\ell,C}P\,dw
 =\int b_Q L^K(P\chi_{\ell,C})\,dw.
\]
By Definition~\ref{def:smooth-L-atom},
\[
 \|b_Q\|_{L^1(dw)}\lesssim r^{2K}w(Q)^{1-1/p}.
\]
Put $U=U_{\ell,C}$ and $D=D_\ell$. Since $U$ has radius
$8J_GD$, the polynomial estimate \eqref{eq:orthonormal-polynomial-bound}
(after adjusting a fixed geometric constant) and
\eqref{eq:Latom-cutoff-bounds} imply
\[
 \|\partial^\alpha(P\chi_{\ell,C})\|_\infty
 \lesssim D^{-|\alpha|}w(U)^{-1/2}
 \|P\|_{L^2(U,dw)},\qquad |\alpha|\le2K.
\]
The function $P\chi_{\ell,C}$ is smooth and compactly supported.
Apply Lemma~\ref{lem:euclidean-to-dunkl} at scale $D$ and use
$L^K=(-1)^K(\sum_jT_j^2)^K$. We obtain the uniform bound
\[
 \|L^K(P\chi_{\ell,C})\|_\infty
 \lesssim D_\ell^{-2K}w(U_{\ell,C})^{-1/2}
 \|P\|_{L^2(U_{\ell,C},dw)}.
\]
The estimate is global because $P\chi_{\ell,C}$ is compactly supported.
Therefore
\[\Big|\int A_{\ell,C}P\,dw\Big|
 \lesssim
 \Big(\frac r{D_\ell}\Big)^{2K}
 \frac{w(Q)^{1-1/p}}{w(U_{\ell,C})^{1/2}}
 \|P\|_{L^2(U_{\ell,C},dw)}.\]
Apply Lemma~\ref{lem:smooth-moment-correction} to obtain
\(q_{\ell,C}\in C_c^\infty(2U_{\ell,C})\) such that
\begin{equation}\label{eq:Latom-corrector-moments}
 \int q_{\ell,C}P\,dw=\int A_{\ell,C}P\,dw,
 \qquad P\in\mathcal P_{L_0},
\end{equation}
and, for \(|\alpha|\leq2M\),
\begin{equation}\label{eq:Latom-corrector-derivatives}
 |\partial^\alpha q_{\ell,C}(x)|
 \lesssim
 \Big(\frac r{D_\ell}\Big)^{2K}
 w(Q)^{1-1/p}w(U_{\ell,C})^{-1}D_\ell^{-|\alpha|}.
\end{equation}

At the initial scale set
\[
 H_{0,C}=A_{0,C}-q_{0,C},\qquad C\in\mathscr C_0.
\]
If \(C\in\mathscr C_{\ell+1}\), let \(\operatorname{ch}_\ell(C)\) denote the
components of \(\mathscr C_\ell\) contained in \(C\), and put
\begin{equation}\label{eq:Latom-transition}
 H_{\ell+1,C}
 =\sum_{C'\in\operatorname{ch}_\ell(C)}q_{\ell,C'}-q_{\ell+1,C}.
\end{equation}
Because of \eqref{eq:Latom-partition-on-support},
\[
 A_{\ell+1,C}
 =\sum_{C'\in\operatorname{ch}_\ell(C)}A_{\ell,C'}.
\]
Hence \eqref{eq:Latom-corrector-moments} implies that every
\(H_{\ell,C}\) has all ordinary moments through degree \(L_0\) equal to zero.
At the last level there is one component \(C_*\), and
\(A_{L_Q,C_*}=A_Q\).  Since $2K>L_0$, repeated Dunkl integration by parts (equivalently, first
insert a $G$-invariant smooth cutoff equal to one on the compact support) gives
\[
 \int A_QP\,dw=\int b_QL^KP\,dw=0,
 \qquad P\in\mathcal P_{L_0},
\]
because $L$ lowers ordinary polynomial degree by two. Hence we may take
\(q_{L_Q,C_*}=0\). Telescoping now gives the finite identity
\begin{equation}\label{eq:Latom-telescoping}
 A_Q
 =\sum_{C\in\mathscr C_0}H_{0,C}
 +\sum_{\ell=0}^{L_Q-1}
 \sum_{C\in\mathscr C_{\ell+1}}H_{\ell+1,C}.
\end{equation}

By \eqref{eq:Latom-cluster-diameter},
\[
 2U_{\ell,C'}\subset 2U_{\ell+1,C}
 \qquad(C'\in\operatorname{ch}_\ell(C)).
\]
Choose an ordinary Euclidean dyadic cube \(R_{\ell,C}\), of one fixed
generation relative to \(D_\ell\), so that
\[
 2U_{\ell,C}\subset3R_{\ell,C},
 \qquad
 \ell(R_{\ell,C})\simeq D_\ell,
 \qquad
 w(R_{\ell,C})\simeq w(U_{\ell,C}).
\]
For \(H_{0,C}\), the derivative bounds of the smooth operator atom and
\eqref{eq:Latom-corrector-derivatives} give the differential-atom bounds at scale
\(r\).  For \(\ell\geq0\), the transition terms in
\eqref{eq:Latom-transition} satisfy, for \(|\alpha|\leq2M\),
\[
 |\partial^\alpha H_{\ell+1,C}(x)|
 \lesssim
 \Big(\frac r{D_\ell}\Big)^{2K}
 w(Q)^{1-1/p}w(U_{\ell+1,C})^{-1}D_\ell^{-|\alpha|}.
\]
Lemma~\ref{lem:euclidean-to-dunkl} also gives the Dunkl derivative
estimates. Normalizing by
$w(R_{\ell+1,C})^{-1/p}$, doubling and $G$-invariance give
\[
 |\lambda_{\ell+1,C}|
 \lesssim
 2^{-\ell[2K-\Qdim(1/p-1)_+]},
\]
while
\[
 |\lambda_{\ell+1,C}|w(R_{\ell+1,C})^{-1/p}
 \lesssim
 2^{-2K\ell}\frac{w(Q)}{w(R_{\ell+1,C})}w(Q)^{-1/p}
 \lesssim 2^{-2K\ell}w(Q)^{-1/p}.
\]
Since $K$ is auxiliary, it may be chosen so large that
$2K-\Qdim(1/p-1)_+>E$ and $2K>E$.  Each level has at most $J_G$ components. Renaming the terms in
\eqref{eq:Latom-telescoping} as \(a_{\ell,h,Q}\) proves
\eqref{eq:Latom-differential-decomposition}--\eqref{eq:Latom-weighted-coefficient}.

This completes our proof.
\end{proof}

Dziuba\'nski and Hejna-{\L}y\.zwa \cite{DziubanskiHejnaLyzwa2026}
obtained Coifman--Weiss atoms for Dunkl $H^p$, $0<p\leq1$.
Here we also require ordinary derivative bounds, higher-order moments,
and Besov--Triebel--Lizorkin coefficient estimates.

\begin{lemma}
\label{lem:L-atomic-transfer}
Fix \(s\in\R\) and \(0<p,q\le\infty\). Choose the auxiliary order \(K\) in
Proposition~\ref{prop:Latom-to-atoms} so that \(E>0\) is sufficiently large.
After grouping atoms associated with the same output cube, the coefficient family \(v\)
obtained from \(\sum_Qu_QA_Q\) satisfies
\begin{align}
 \|v\|_{\dot b^{s,\mathrm{at}}_{p,q}}
 &\lesssim \|u\|_{\dot b^{s,\mathrm{at}}_{p,q}},
 \label{eq:L-transfer-b}\\
 \|v\|_{\dot f^{s,\mathrm{at}}_{p,q}}
 &\lesssim \|u\|_{\dot f^{s,\mathrm{at}}_{p,q}},
 \label{eq:L-transfer-f}
\end{align}
with the second estimate valid when $p<\infty$.
\end{lemma}

	\begin{proof}
		Let $Q\in\Dcal_\nu$ be an input cube associated
		with a smooth $L$-atom $A_Q$. By
		Proposition~\ref{prop:Latom-to-atoms}, for each
		$\ell\ge0$, the atom $A_Q$ produces differential
		atoms $a_{\ell,h,Q}$ supported in $3R_{\ell,h,Q}$.
		Let $R=R_{\ell,h,Q}\in\Dcal_\mu$ be one of these
		output cubes.
		
		By \eqref{eq:Latom-geometry}, there exist
		two fixed nonnegative integers $a^-,a^+$
		such that
		\[
		2^{\ell-a^-}\ell(Q)
		\le \ell(R)
		\le 2^{\ell+a^+}\ell(Q).
		\]
		Hence
		\begin{equation}\label{eq:transfer-generations}
		\nu-\ell-a^+\le\mu\le\nu-\ell+a^-.
		\end{equation}
		In particular, for fixed $\nu$ and $\ell$,
		there are only a fixed number of possible
		output generations.

	Fix $\nu$ and $\ell\ge0$. By
	\eqref{eq:Latom-telescoping}, each $Q\in\Dcal_\nu$
	produces at most $|G|$ output terms at level
	$\ell$. Conversely, if $Q$ contributes to $R$, then
	\eqref{eq:Latom-geometry} gives
	\[
	x_Q\in\bigcup_{\sigma\in G}
	B(\sigma(x_R),C2^\ell2^{-\nu}).
	\]
	Hence at most $C_G2^{\ell N}$ input cubes
	contribute to a fixed $R$.
	
	Let $v_R^{\nu,\ell}$ be the sum of these
	contributions. The
	$p$-triangle inequality and H\"older\'s inequality give
	\[
	\left(
	\sum_{R\in\Dcal_\mu}
	|v_R^{\nu,\ell}|^p
	\right)^{1/p}
	\lesssim
	2^{-\ell[E-N(1-1/p)_+]}
	\left(
	\sum_{Q\in\Dcal_\nu}|u_Q|^p
	\right)^{1/p},
	\]
	with the usual modification for $p=\infty$.
	
	Since
	\[
	2^{\mu s}
	\lesssim 2^{\nu s}2^{\ell|s|},
	\]
	and there are at most $a^-+a^++1$
	possible values of $\mu$, summing over
	$\nu,\ell$ proves \eqref{eq:L-transfer-b},
	provided
	\[
	E>N(1-1/p)_++|s|.
	\]
	
	For \eqref{eq:L-transfer-f}, put
	\[
	G_\nu
	=\sum_{Q\in\Dcal_\nu}
	|u_Q|w(Q)^{-1/p}\1_Q
	\]
	and choose
	\[
	0<r_0<\min\{1,p,q\},
	\]
	with $r_0<\min\{1,p\}$ when $q=\infty$.
	
	If $Q$ contributes to $R$ and $x\in R$,
	then
	\[
	Q\subset B^d(x,C2^\ell2^{-\nu}).
	\]
	By doubling and $G$-invariance,
	\[
	w(B^d(x,C2^\ell2^{-\nu}))
	\lesssim 2^{\ell\Qdim}w(Q).
	\]
	Thus \eqref{eq:Latom-weighted-coefficient}
	gives
	\begin{align*}
		\sum_{R\in\Dcal_\mu}
		|v_R^{\nu,\ell}|
		w(R)^{-1/p}\1_R(x)
		&\quad\lesssim 2^{-\ell E}
		\sum_{\substack{Q\in\Dcal_\nu\\
				Q\subset B^d(x,C2^\ell2^{-\nu})}}
		|u_Q|w(Q)^{-1/p}\\
		&\quad\lesssim
		2^{-\ell(E-\Qdim/r_0)}
		\MO_{r_0}G_\nu(x).
	\end{align*}
	Consequently,
	\begin{align*}
		&2^{\mu s}
		\sum_{R\in\Dcal_\mu}
		|v_R^{\nu,\ell}|
		w(R)^{-1/p}\1_R(x)
		 \lesssim
		2^{\nu s}
		2^{-\ell(E-\Qdim/r_0-|s|)}
		\MO_{r_0}G_\nu(x).
	\end{align*}
	Choose
	\[
	E>\Qdim/r_0+|s|.
	\]
	By \eqref{eq:transfer-generations},
	each output generation receives
	contributions from only a fixed number
	of input generations for each $\ell$.
	Summing over $\ell$ and applying the
	vector-valued Fefferman--Stein inequality
	give
	\[
	\|v\|_{\dot f^{s,\mathrm{at}}_{p,q}}
	\lesssim
	\left\|
	\left(
	\sum_{\nu\in\mathbb Z}
	[2^{\nu s}G_\nu]^q
	\right)^{1/q}
	\right\|_{L^p(dw)}
	=
	\|u\|_{\dot f^{s,\mathrm{at}}_{p,q}},
	\]
	with the usual modification for $q=\infty$.
	
	This completes our proof.
\end{proof}

We finish the proof of Theorem~\ref{thm:atomic-characterization}.

\begin{proof}
The direct inclusion follows from
\eqref{eq:atom-synthesis-B}--\eqref{eq:atom-synthesis-F}.
For the converse, apply Lemma~\ref{lem:smooth-L-atomic-decomp},
Proposition~\ref{prop:Latom-to-atoms}, and
Lemma~\ref{lem:L-atomic-transfer}, taking $K$ sufficiently large.
They give $f=\sum_R v_Ra_R$ in $\Scirc'$ with differential atoms $a_R$ and
\begin{align*}
 \|v\|_{\dot b^{s,\mathrm{at}}_{p,q}}
 &\lesssim \|u\|_{\dot b^{s,\mathrm{at}}_{p,q}}
 \lesssim\|f\|_{\dot B^s_{p,q}(dw)},\\
 \|v\|_{\dot f^{s,\mathrm{at}}_{p,q}}
 &\lesssim \|u\|_{\dot f^{s,\mathrm{at}}_{p,q}}
 \lesssim\|f\|_{\dot F^s_{p,q}(dw)},
\end{align*}
for the respective spaces. The wavelet norm equivalences follow from
Theorem~\ref{thm:frame-spaces}.

This completes our proof.
\end{proof}

\section{Applications}

We give two applications of the wavelet frame constructed
in Section~3. First, we establish fractional chain rules
for Dunkl Besov and Triebel--Lizorkin spaces.
We then study nonlinear wavelet approximation and prove
a Jackson inequality.

\subsection{Fractional chain rules}
Following \cite{BuiBuiDuongChain}, let $F:\C\to\C$ and $G:\C\to[0,\infty)$ satisfy
\begin{equation}\label{eq:chain-condition}
 |F(z)-F(w)|\leq [G(z)+G(w)]|z-w|,
 \qquad z,w\in\C.
\end{equation}
For homogeneous spaces, $F(u)$ is understood modulo $F(0)$.
Typically, $G(z)\simeq|F'(z)|$.

\begin{theorem}\label{thm:chain-F}
Let $0<s<1$, $0<p<\infty$, $1<p_1\leq\infty$, $0<p_2<\infty$, and $0<q\leq\infty$ satisfy
\[
 \frac1p=\frac1{p_1}+\frac1{p_2}.
\]
Assume either
\begin{equation}\label{eq:F-chain-condition-indices}
 0\leq
 \Qdim\Big(
 \max\Big\{\frac1{p_2},\frac1q,1\Big\}
 -\frac{p_1-1}{p_1}
 \Big)<s,
\end{equation}
or $\min\{p,q\}>1$. Let $u\in\Sscr(\R^N)$. If $G(u)\in L^{p_1}(dw)$ and \eqref{eq:chain-condition} holds, then
\[\norm{F(u)}_{\dotF^{s}_{p,q}(dw)}
 \lesssim
 \norm{G(u)}_{L^{p_1}(dw)}
 \norm{u}_{\dotF^{s}_{p_2,q}(dw)}.\]
The same estimate holds for the operator-associated Triebel--Lizorkin spaces.
\end{theorem}

\begin{theorem}\label{thm:chain-B}
Let $0<s<1$, $0<p<\infty$, $1<p_1\leq\infty$, $0<p_2,q\leq\infty$, and
\[
 \frac1p=\frac1{p_1}+\frac1{p_2}.
\]
Assume either
\begin{equation}\label{eq:B-chain-condition-indices}
 0\leq
 \Qdim\Big(
 \max\Big\{\frac1{p_2},1\Big\}-\frac{p_1-1}{p_1}
 \Big)<s,
\end{equation}
or $p>1$. Let $u\in\Sscr(\R^N)$. If $G(u)\in L^{p_1}(dw)$ and \eqref{eq:chain-condition} holds, then
\[\norm{F(u)}_{\dotB^{s}_{p,q}(dw)}
 \lesssim
 \norm{G(u)}_{L^{p_1}(dw)}
 \norm{u}_{\dotB^{s}_{p_2,q}(dw)}.\]
The same estimate holds for the operator-associated Besov spaces.
\end{theorem}

For this section put
\[
D_j:=P_j^2.
\]
By \eqref{eq:square-partition},
\[u=\sum_{j\in\Z}D_ju\]
in $L^2$ and in $\Scirc'$. Let
\[
A_j u(x):=
\sum_{Q\in\Dcal_{j+J}}w(Q)^{-1/2}
|\ip{u}{\psi_Q}|\1_Q(x).
\]
By Theorem~\ref{thm:frame-spaces}, the norms are equivalent to those
of $\{2^{js}A_ju\}$.

\begin{lemma}\label{lem:Dj-frame-control}
Let $0<r\leq1$. Then
\begin{equation}\label{eq:D-size-frame}
 |D_ju(x)|\lesssim \MO_r(A_ju)(x),
\end{equation}
where $\MO_rh=[\MO(|h|^r)]^{1/r}$. Moreover, whenever $|x-y|\leq c2^{-j}$,
\begin{equation}\label{eq:D-local-one-sided}
 |D_ju(x)-D_ju(y)|
 \lesssim 2^j|x-y|\,\MO_r(A_ju)(x),
\end{equation}
and the same estimate holds with $x$ and $y$ interchanged. Consequently,
\begin{equation}\label{eq:D-difference-frame}
 |D_ju(x)-D_ju(y)|
 \lesssim \min\{1,2^j|x-y|\}
 \big[\MO_r(A_ju)(x)+\MO_r(A_ju)(y)\big].
\end{equation}
\end{lemma}

\begin{proof}
By \eqref{eq:kernel-frame},
\[
 D_ju=\sum_{Q\in\Dcal_{j+J}}\ip{u}{\psi_Q}\psi_Q.
\]
Using \eqref{eq:frame-size} and Lemma~\ref{lem:discrete} with
$b_Q=w(Q)^{-1/2}|\ip{u}{\psi_Q}|$ gives
\eqref{eq:D-size-frame}.

For $|x-y|\leq c2^{-j}$, \eqref{eq:frame-reg} gives
\begin{align*}
 |D_ju(x)-D_ju(y)|
 &\lesssim2^j|x-y|
 \sum_Qw(Q)^{-1/2}|\ip{u}{\psi_Q}|\\
 &\qquad\times\big[(1+2^jd(x,x_Q))^{-M}
 +(1+2^jd(y,x_Q))^{-M}\big].
\end{align*}
As $d(x,y)\lesssim2^{-j}$, the two weights are comparable. Apply
Lemma~\ref{lem:discrete} and interchange $x,y$ to obtain
\eqref{eq:D-local-one-sided}. For $|x-y|>c2^{-j}$ use
\eqref{eq:D-size-frame}; together these give \eqref{eq:D-difference-frame}.

This completes our proof.
\end{proof}

\begin{lemma}\label{lem:chain-kernel-tools}
Let $D_j(x,y)$ be the kernel of $D_j=P_j^2$. For every $M>0$,
\begin{equation}\label{eq:Dj-kernel-majorant}
 |D_j(x,y)|\lesssim
 \frac{(1+2^jd(x,y))^{-M}(1+2^j|x-y|)^{-2}}
 {w(B(x,2^{-j}+d(x,y)))}.
\end{equation}
For every nonnegative locally integrable $H$,
\begin{align}
 \int |D_j(x,y)|H(y)\,dw(y)&\lesssim\MO H(x),
 \label{eq:Dj-max-int}\\
 \int 2^j|x-y||D_j(x,y)|H(y)\,dw(y)&\lesssim\MO H(x).
 \label{eq:Dj-moment-int}
\end{align}
If $k\geq j$, $Q\in\Dcal_{k+J}$, $1\leq\beta<\infty$, and $M_0>0$ is fixed in advance, then, for $M$ large enough,
\begin{align}
 &\Big[\int |D_j(x,y)|(1+2^kd(y,x_Q))^{-M\beta}\,dw(y)\Big]^{1/\beta}
 \notag\\
 &\quad\lesssim
 \Big[\frac{w(Q)}{w(B(x_Q,2^{-j}))}\Big]^{1/\beta}
 (1+2^jd(x,x_Q))^{-M_0}.
 \label{eq:coarse-fine-convolution}
\end{align}
In particular, for $1<\theta<\infty$ we may take $\beta=\theta'=\theta/(\theta-1)$, as well as the endpoint $\beta=1$.
\end{lemma}

\begin{proof}
Proposition~\ref{prop:kernel} gives
\eqref{eq:Dj-kernel-majorant}. For \eqref{eq:Dj-max-int}, put $t=2^{-j}$ and decompose $\R^N$ into the orbit annuli
\[
 E_0=B^d(x,2t),\qquad
 E_\ell=B^d(x,2^{\ell+1}t)\setminus B^d(x,2^\ell t),\quad \ell\geq1.
\]
On $E_\ell$, \eqref{eq:Dj-kernel-majorant}, doubling, and the comparison of
Euclidean and orbit balls give, after choosing the orbit-decay order larger
than $\Qdim+2$,
\[
 \int_{E_\ell}|D_j(x,y)|H(y)\,dw(y)
 \lesssim 2^{-\ell(M-\Qdim)}
 \frac1{w(B^d(x,C2^{\ell+1}t))}
 \int_{B^d(x,C2^{\ell+1}t)}H(y)\,dw(y).
\]
Summing in $\ell$ proves \eqref{eq:Dj-max-int}.  For
\eqref{eq:Dj-moment-int} we use
\[
 (|x-y|/t)(1+|x-y|/t)^{-2}\lesssim1
\]
and repeat the same orbit-annular argument with a slightly larger value of
$M$.

For \eqref{eq:coarse-fine-convolution}, put
$t_j=2^{-j}$, $t_k=2^{-k}$ and $z=x_Q$; thus $t_k\leq t_j$.  We claim first
that, for arbitrary $A_0>0$, after taking the two decay exponents in
\eqref{eq:Dj-kernel-majorant} and in the second factor sufficiently large,
\begin{equation}\label{eq:coarse-fine-L1}
 \int |D_j(x,y)|(1+d(y,z)/t_k)^{-M\beta}\,dw(y)
 \lesssim
 \frac{w(B(z,t_k))}{w(B(z,t_j))}
 (1+d(x,z)/t_j)^{-A_0\beta}.
\end{equation}
Write
\[
 a=\frac{d(x,y)}{t_j},\qquad
 b=\frac{d(y,z)}{t_k},\qquad
 \delta=\frac{d(x,z)}{t_j}.
\]
Because $t_k\leq t_j$ and $d$ satisfies the triangle inequality,
$1+\delta\lesssim(1+a)(1+b)$.  Hence, for any prescribed $A_0$,
\[
 (1+a)^{-A}(1+b)^{-M\beta}
 \lesssim (1+\delta)^{-A_0\beta}
 (1+a)^{-(A-A_0\beta)}
 (1+b)^{-(M-A_0)\beta}.
\]
We move the volume term from $x$ to $z$.  Put $R_x=t_j+d(x,y)$. The orbit triangle inequality gives
$d(x,z)\leq d(x,y)+d(y,z)$. Choose $\sigma\in G$ such that
$|x-\sigma z|=d(x,z)$. By $G$-invariance of $dw$,
\[
 w(B(z,t_j))=w(B(\sigma z,t_j))
 \leq w\bigl(B(x,R_x+d(y,z))\bigr).
\]
Doubling, applied at the fixed center $x$, now gives
\[
 \frac1{w(B(x,R_x))}
 \lesssim
 \frac{(1+d(y,z)/R_x)^{\Qdim}}{w(B(z,t_j))}
 \leq
 \frac{(1+b)^{\Qdim}}{w(B(z,t_j))}.
\]
Here $G$-invariance is used to compare reflected centers.
After these two reductions, the left-hand side of
\eqref{eq:coarse-fine-L1} is bounded by
\[
 \frac{(1+\delta)^{-A_0\beta}}{w(B(z,t_j))}
 \int (1+d(y,z)/t_k)^{-H}\,dw(y)
\]
with $H>\Qdim$ if $A$ and $M$ were chosen sufficiently large.  Decomposing
into the orbit annuli centered at $z$ at scale $t_k$ yields
\[
 \int (1+d(y,z)/t_k)^{-H}\,dw(y)\lesssim w(B^d(z,t_k))
 \simeq w(B(z,t_k)),
\]
which proves \eqref{eq:coarse-fine-L1}.

Finally, $w(B(z,t_k))\simeq_J w(Q)$.  Taking the $\beta$-th root in
\eqref{eq:coarse-fine-L1}, with $A_0=M_0$ before
renaming the arbitrary decay order, proves
\eqref{eq:coarse-fine-convolution}.  The estimate also holds for $\beta=1$.

This completes our proof.
\end{proof}

We use the following dyadic estimate, cf.~\cite{BuiBuiDuongChain}.

\begin{lemma}\label{lem:nonlinear-dyadic}
Assume \eqref{eq:chain-condition}. Let $0<r\leq1$ and $1<\theta<\infty$. Then, for Schwartz $u$ and every $j\in\Z$,
\begin{align}
 |D_j(F(u))(x)|
 &\lesssim \MO(G(u))(x)
 \sum_{k<j}2^{-(j-k)}\MO_r(A_ku)(x)\notag\\
 &\quad+
 \MO_\theta(G(u))(x)
 \sum_{k\geq j}
 2^{\Qdim(k-j)(1/r-1+1/\theta)}
 \MO_r(A_ku)(x),
 \label{eq:nonlinear-dyadic}
\end{align}
where $\MO_\theta h=[\MO(|h|^\theta)]^{1/\theta}$.
\end{lemma}

\begin{proof}
Since $D_j$ has cancellation,
\begin{align}
 |D_j(F(u))(x)|
 &\leq G(u)(x)\int |D_j(x,y)||u(y)-u(x)|\,dw(y)\notag\\
 &\quad+\int |D_j(x,y)|G(u)(y)|u(y)-u(x)|\,dw(y)\notag\\
 &=:E_j(x)+H_j(x). \label{eq:EjHj}
\end{align}
Insert $u=\sum_kD_ku$. We estimate $H_j$; the bound for $E_j$
follows similarly.

\smallskip
\noindent\emph{Low frequencies: $k<j$.}
Split the $y$-integral into
$\Omega_1=\{|x-y|\leq2^{-k}\}$ and
$\Omega_2=\{|x-y|>2^{-k}\}$.
On $\Omega_1$, the one-sided estimate \eqref{eq:D-local-one-sided} at scale
$k$ gives
\[
 |D_ku(y)-D_ku(x)|
 \lesssim2^k|x-y|\,\MO_r(A_ku)(x).
\]
Hence \eqref{eq:Dj-moment-int} yields
\begin{align}
 &\int_{\Omega_1}|D_j(x,y)|G(u)(y)|D_ku(y)-D_ku(x)|\,dw(y)\notag\\
 &\qquad\lesssim2^{k-j}\MO_r(A_ku)(x)\MO(G(u))(x).
 \label{eq:low-near}
\end{align}

On $\Omega_2$, use
$|D_ku(y)-D_ku(x)|\leq|D_ku(y)|+|D_ku(x)|$.
For $D_ku(x)$, \eqref{eq:D-size-frame}--\eqref{eq:Dj-max-int} apply, and
\[
 (1+2^j|x-y|)^{-2}\lesssim2^{-2(j-k)}
 \qquad (y\in\Omega_2),
\]
so that contribution is at most
\[
 C2^{-2(j-k)}\MO_r(A_ku)(x)\MO(G(u))(x).
\]
For $D_ku(y)$, use
\[
 |D_ku(y)|\lesssim\sum_{Q\in\Dcal_{k+J}}b_Q
 (1+2^kd(y,x_Q))^{-M},
 \qquad b_Q=w(Q)^{-1/2}|\ip{u}{\psi_Q}|.
\]
By the triangle inequality for $d$,
\[
 (1+2^kd(y,x_Q))^{-M}
 \lesssim (1+2^kd(x,x_Q))^{-M/2}(1+2^kd(x,y))^{M/2}.
\]
Choose the orbit-decay exponent in \eqref{eq:Dj-kernel-majorant} much larger
than $M$. Since $k<j$, the factor $(1+2^kd(x,y))^{M/2}$ is absorbed by the
scale-$j$ orbit decay. On $\Omega_2$ the Euclidean factor contributes $2^{-2(j-k)}$;
hence
\begin{align*}
 &\int_{\Omega_2}|D_j(x,y)|G(u)(y)|D_ku(y)|\,dw(y)\\
 &\quad\lesssim2^{-2(j-k)}\MO(G(u))(x)
 \sum_Qb_Q(1+2^kd(x,x_Q))^{-M/2}\\
 &\quad\lesssim2^{-2(j-k)}\MO(G(u))(x)\MO_r(A_ku)(x),
\end{align*}
where the last step is Lemma~\ref{lem:discrete}. Together with
\eqref{eq:low-near}, this gives the first line of
\eqref{eq:nonlinear-dyadic} after summing over $k<j$.

\smallskip
\noindent\emph{Fine frequencies: $k\geq j$.}
The piece containing $D_ku(x)$ satisfies
\[
 |D_ku(x)|\int|D_j(x,y)|G(u)(y)\,dw(y)
 \lesssim\MO_r(A_ku)(x)\MO(G(u))(x)
\]
by \eqref{eq:D-size-frame} and \eqref{eq:Dj-max-int}. Since
$1/r-1+1/\theta>0$, this is dominated by the corresponding term in the second
line of \eqref{eq:nonlinear-dyadic}.

For $D_ku(y)$, apply the frame expansion, H\"older's inequality with
$\theta,\theta'$, and Lemma~\ref{lem:chain-kernel-tools}:
\begin{align*}
 &\int |D_j(x,y)|G(u)(y)(1+2^kd(y,x_Q))^{-M}\,dw(y)\\
 &\quad\lesssim\MO_\theta(G(u))(x)
 \Big[\frac{w(Q)}{w(B(x_Q,2^{-j}))}\Big]^{1/\theta'}
 (1+2^jd(x,x_Q))^{-M_0}.
\end{align*}
Summing over $Q$ and applying Lemma~\ref{lem:discrete} with
$\lambda=1/\theta'=1-1/\theta$ yields
\begin{align}
 &\int |D_j(x,y)|G(u)(y)|D_ku(y)|\,dw(y)\notag\\
 &\qquad\lesssim
 2^{\Qdim(k-j)(1/r-1+1/\theta)}
 \MO_\theta(G(u))(x)\MO_r(A_ku)(x).
 \label{eq:fine-H}
\end{align}
For $E_j$, the same calculation is performed without the factor $G(u)(y)$;
using the endpoint $\beta=1$ of
\eqref{eq:coarse-fine-convolution} gives the smaller factor $2^{\Qdim(k-j)(1/r-1)}$, which is absorbed by the right-hand side
of \eqref{eq:fine-H}. This proves \eqref{eq:nonlinear-dyadic}.

For the Banach range, keep $G(u)$ and the scale-$k$ term together
and take $r=1$. The fine-frequency term is bounded by
\begin{equation}\label{eq:nonlinear-good-range}
 \sum_{k\geq j}
 \big[\MO(G(u)A_ku)(x)+G(u)(x)\MO(A_ku)(x)\big].
\end{equation}
After multiplication by $2^{js}$, the sum $k\geq j$ has the summable factor
$2^{-s(k-j)}$. This is the Banach-range form used below.

This completes our proof.
\end{proof}

\begin{proof}
Assume first \eqref{eq:F-chain-condition-indices}. We may choose
\[
 0<r<\min\{p_2,q,1\},\qquad1<\theta<p_1
\]
(with $\theta$ arbitrarily large if $p_1=\infty$) so close to the endpoints
that
\[\alpha:=s-\Qdim(1/r-1+1/\theta)>0.\]
Multiplying \eqref{eq:nonlinear-dyadic} by $2^{js}$ gives the scale factors
\[
 2^{-(1-s)(j-k)}\quad(k<j),
 \qquad
 2^{-\alpha(k-j)}\quad(k\geq j).
\]
Discrete convolution in $\ell^q$ and H\"older's inequality give
\begin{align*}
 &\Big\|\Big(\sum_j[2^{js}|D_jF(u)|]^q\Big)^{1/q}\Big\|_p\\
 &\quad\lesssim\|G(u)\|_{p_1}
 \Big\|\Big(\sum_k[2^{ks}\MO_r(A_ku)]^q\Big)^{1/q}\Big\|_{p_2}.
\end{align*}
Since $r<\min\{p_2,q\}$ when $q<\infty$, Fefferman--Stein gives
\[
 \Big\|\Big(\sum_k[2^{ks}\MO_r(A_ku)]^q\Big)^{1/q}\Big\|_{p_2}
 \lesssim
 \Big\|\Big(\sum_k[2^{ks}A_ku]^q\Big)^{1/q}\Big\|_{p_2}.
\]
Theorem~\ref{thm:frame-spaces} identifies the last quantity with
$\|u\|_{\dotF^{s}_{p_2,q}(dw)}$, while the cutoff-independence argument from
\cite[Theorem~3.4]{BuiBuiDuongFMS}, applied to the annular partition
$D_j=P_j^2$, identifies the left side with
$\|F(u)\|_{\dotF^{s}_{p,q}(dw)}$.

If $\min\{p,q\}>1$, then $p_2>1$. By the Banach-range estimate
\eqref{eq:nonlinear-good-range}, the scale kernels are
$2^{-(1-s)(j-k)}$ and $2^{-s(k-j)}$. The vector-valued maximal
inequality and H\"older's inequality complete the proof for $0<s<1$.

This completes our proof.
\end{proof}

\begin{proof}
Assume first \eqref{eq:B-chain-condition-indices}. Choose
\[
 0<r<\min\{p_2,1\},\qquad1<\theta<p_1
\]
so that $\alpha=s-\Qdim(1/r-1+1/\theta)>0$. After multiplication by
$2^{js}$, Lemma~\ref{lem:nonlinear-dyadic} produces the summable scale kernels
\[
 2^{-(1-s)(j-k)}\1_{k<j},\qquad
 2^{-\alpha(k-j)}\1_{k\geq j}.
\]
Take the $L^p$ quasi-norm first. H\"older's inequality, boundedness of
$\MO_r$ on $L^{p_2}$ and of $\MO_\theta$ on $L^{p_1}$, and then the discrete
convolution inequality in $\ell^q$ give
\[
 \|F(u)\|_{\dotB^{s}_{p,q}(dw)}
 \lesssim\|G(u)\|_{p_1}
 \Big(\sum_k[2^{ks}\|A_ku\|_{p_2}]^q\Big)^{1/q}.
\]
The final factor is equivalent to $\|u\|_{\dotB^{s}_{p_2,q}(dw)}$ by
Theorem~\ref{thm:frame-spaces}.

If $p>1$, then $p_2>1$. Use \eqref{eq:nonlinear-good-range} and the
summable factors $2^{-(1-s)(j-k)}$, $2^{-s(k-j)}$. The scalar maximal
inequality and H\"older's inequality complete the proof.

This completes our proof.
\end{proof}

\begin{remark}
In the first-order comparison range of \cite[Theorem~1.1]{BuiDunkl}, Theorems~\ref{thm:chain-F} and~\ref{thm:chain-B} also give chain rules on the spaces of homogeneous type $(\R^N,|\cdot|,dw)$. Outside that first-order range the statements still make sense in the transform/operator scale, which is defined for all $s$ and for the admissible ranges of $p,q$.

The lower-smoothness condition above is the one obtained directly from the wavelet proof, where the two-scale discrete estimate uses $0<r\leq1$. In the overlap with the homogeneous-type theorem of \cite{BuiBuiDuongChain}, one may use the sharper admissible condition established there. We keep the present proof self-contained and do not use that sharper condition here.
\end{remark}

\subsection{Nonlinear wavelet approximation}\label{subsec:nonlinear-approx}
We prove a Jackson estimate for the frame $\Psi$, using the local volumes $w(Q)$.

Let $s>0$, $1\le p<\infty$, and choose $0<\tau<p$ so that
\begin{equation}\label{eq:jackson-indices}
 \frac1\tau=\frac{s}{\Qdim}+\frac1p.
\end{equation}
Define the volume-adapted \emph{frame coefficient} Besov space
$\widetilde{\dot B}^{s}_{\tau,\tau}(dw;\Psi)$ to consist of
$f\in\Scirc'$ such that
\begin{equation}\label{eq:jackson-besov-norm}
 \|f\|_{\widetilde{\dot B}^{s}_{\tau,\tau}(dw;\Psi)}
 :=\bigg(\sum_{j\in\Z}\sum_{Q\in\Dcal_{j+J}}
 \big[ w(Q)^{-s/\Qdim+1/\tau-1/2}
 |\ip{f}{\psi_Q}|\big]^\tau\bigg)^{1/\tau}<\infty.
\end{equation}
Since $w(Q)\simeq w(B(x_Q,2^{-j}))$, the factor $w(Q)^{-s/\Qdim}$
measures smoothness through local volume.  In view of
\eqref{eq:jackson-indices}, the norm in \eqref{eq:jackson-besov-norm}
can also be written as
\begin{equation}\label{eq:jackson-besov-coeff}
 \|f\|_{\widetilde{\dot B}^{s}_{\tau,\tau}(dw;\Psi)}
 =\bigg(\sum_{j,Q}
 [|\ip{f}{\psi_Q}|w(Q)^{1/p-1/2}]^\tau\bigg)^{1/\tau}.
\end{equation}
We do not claim an intrinsic Littlewood--Paley characterization.

For $n\ge1$ set
\[
 \sigma_n(f)_p:=\inf_{\substack{\Lambda\subset\Psi,\,\#\Lambda\le n\\
                  (a_\psi)_{\psi\in\Lambda}\subset\C}}
 \Big\|f-\sum_{\psi\in\Lambda}a_\psi\psi\Big\|_{L^p(dw)}.
\]
We adapt \cite[Lemma~6.17]{KerkyacharianPetrushev} to the Dunkl frame.

\begin{lemma}\label{lem:jackson-finite-sum}
Let $1\le p<\infty$. For any finite set $E$ of frame indices and scalars
$(a_Q)_{Q\in E}$,
\begin{equation}\label{eq:jackson-finite-sum}
 \Big\|\sum_{Q\in E}a_Q\psi_Q\Big\|_{L^p(dw)}
 \lesssim (\#E)^{1/p}
 \sup_{Q\in E}\big\{|a_Q|w(Q)^{1/p-1/2}\big\}.
\end{equation}
The constant is independent of $E$ and $a$.
\end{lemma}

\begin{proof}
We use the Euclidean Hardy--Littlewood maximal operator
$\M_t h=[\M(|h|^t)]^{1/t}$, with $0<t<1$.
By \eqref{eq:frame-size} and the finiteness of $G$, for every sufficiently
large $M$,
\[
 |\psi_Q(x)|\lesssim w(Q)^{-1/2}
 \sum_{\sigma\in G}(1+2^j|x-\sigma x_Q|)^{-M}.
\]
For $Q\in\Dcal_{j+J}$, measure doubling and $G$-invariance show that
\begin{equation}\label{eq:jackson-max-domination}
 (1+2^j|x-\sigma x_Q|)^{-M}
 \lesssim\M_t(\1_{\sigma Q})(x),
 \qquad M>\Qdim/t.
\end{equation}
Indeed, a ball of radius $\simeq|x-\sigma x_Q|+2^{-j}$ contains
$\sigma Q$, and doubling bounds its measure by
$C(1+2^j|x-\sigma x_Q|)^{\Qdim}w(Q)$.
Consequently, the Fefferman--Stein vector-valued maximal inequality
(with exponents $p/t>1$ and $1/t>1$) gives
\begin{equation}\label{eq:jackson-vector}
 \Big\|\sum_{Q\in E}a_Q\psi_Q\Big\|_p
 \lesssim\sum_{\sigma\in G}
 \Big\|\sum_{Q\in E}
 |a_Q|w(Q)^{-1/2}\1_{\sigma Q}\Big\|_p.
\end{equation}
We claim that, for each $\sigma\in G$,
\begin{equation}\label{eq:jackson-nested}
 \Big\|\sum_{Q\in E}w(Q)^{-1/p}\1_{\sigma Q}\Big\|_p
 \lesssim (\#E)^{1/p}.
\end{equation}
The cubes $\sigma Q$ form a nested grid.
For dyadic cubes $R\subset Q$, the explicit volume estimate
\eqref{eq:volume} yields
\[
 \frac{w(R)}{w(Q)}\lesssim
 \Big(\frac{\ell(R)}{\ell(Q)}\Big)^N.
\]
The same holds for their images under $\sigma$, by $G$-invariance.
The selected cubes containing $x$ are nested. Thus
\[
 \sum_{Q\in E}w(Q)^{-1/p}\1_{\sigma Q}(x)
 \lesssim \sup_{Q\in E}w(Q)^{-1/p}\1_{\sigma Q}(x).
\]
Integrating its $p$th power and using
$w(\sigma Q)=w(Q)$ proves \eqref{eq:jackson-nested}.
Finally insert
$|a_Q|w(Q)^{-1/2}\le K w(Q)^{-1/p}$, with
$K=\sup_{Q\in E}|a_Q|w(Q)^{1/p-1/2}$,
into \eqref{eq:jackson-vector} and apply \eqref{eq:jackson-nested}.

This completes our proof.
\end{proof}

\begin{theorem}[Jackson estimate]\label{thm:jackson}
Let $s>0$, $1\le p<\infty$, and
$1/\tau=s/\Qdim+1/p$. Then
$\widetilde{\dot B}^{s}_{\tau,\tau}(dw;\Psi)$ embeds continuously into
$L^p(dw)$ and, for every $f$ in this space and $n\ge1$,
\begin{equation}\label{eq:jackson}
 \sigma_n(f)_p\lesssim n^{-s/\Qdim}
 \|f\|_{\widetilde{\dot B}^{s}_{\tau,\tau}(dw;\Psi)}.
\end{equation}
More precisely, \eqref{eq:jackson} holds for the approximation obtained
by retaining the $n$ largest quantities
$|\ip{f}{\psi_Q}|w(Q)^{1/p-1/2}$.
\end{theorem}

\begin{proof}
Write $c_Q=\ip{f}{\psi_Q}$ and
$d_Q=|c_Q|w(Q)^{1/p-1/2}$. By
\eqref{eq:jackson-besov-coeff}, $d=\{d_Q\}\in\ell^\tau$.
First, Lemma~\ref{lem:jackson-finite-sum} and a dyadic-block
rearrangement (with block sizes $1,2,4,\ldots$) imply, for every
finite family $E$,
\[
 \Big\|\sum_{Q\in E}c_Q\psi_Q\Big\|_p
 \lesssim\Big(\sum_{Q\in E}d_Q^\tau\Big)^{1/\tau},
 \qquad 0<\tau<p.
\]
Hence the frame series converges unconditionally in $L^p(dw)$.
Reconstruction in Theorem~\ref{thm:frame-spaces} identifies its sum
with $f$ in $\Scirc'$.

Let $(d_m^*)_{m\ge1}$ be the nonincreasing rearrangement of $d$.
Then $d_m^*\le m^{-1/\tau}\|d\|_{\ell^\tau}$.
Let $G_n$ be the canonical frame sum over the $n$ largest $d_Q$.
Divide the remaining terms into successive blocks of cardinalities
$n,2n,4n,\ldots$. Lemma~\ref{lem:jackson-finite-sum} gives
\begin{align*}
 \|f-G_n\|_p
 &\lesssim\sum_{k\ge0}(2^kn)^{1/p}d_{2^kn+1}^*\\
 &\lesssim\|d\|_{\ell^\tau}
 n^{1/p-1/\tau}
 \sum_{k\ge0}2^{k(1/p-1/\tau)}\\
 &\lesssim n^{-s/\Qdim}
 \|f\|_{\widetilde{\dot B}^{s}_{\tau,\tau}(dw;\Psi)}.
\end{align*}
Here $\tau<p$.

This completes our proof.
\end{proof}

\begin{remark}
When $\kappa\equiv0$ one has $w(Q)\simeq 2^{-jN}$, so
\eqref{eq:jackson-besov-norm} reduces to the classical wavelet
coefficient norm of $\dot B^s_{\tau,\tau}$.
For general Dunkl weights, the volume factor depends on the location of
$Q$ and cannot simply be replaced by $2^{js}$.
We make no identification with an intrinsic weighted multiplier space.
\end{remark}

\appendix

\section{Inhomogeneous variants}
For completeness we record that the construction has an inhomogeneous version. Choose real-valued functions $\Phi_0,\Phi\in C_c^\infty([0,\infty))$ whose even extensions to $\mathbb R$ are smooth, such that
\[
 \supp\Phi_0\subset[0,2],\qquad
 \supp\Phi\subset[1/2,2],
\]
and
\[
 \Phi_0(t)^2+\sum_{j\geq1}\Phi(2^{-j}t)^2=1,
 \qquad t\geq0.
\]
Let $P_0$ and $P_j$, $j\geq1$, be the corresponding Dunkl multipliers. Apply the cubature theorem to the low-frequency kernel of $P_0$ and to the annular kernels at each $j\geq1$. This gives a self-dual wavelet frame
\[
 \Psi^{\mathrm{inh}}=
 \{\psi_Q^{(0)}:Q\in\Dcal_J\}
 \cup
 \{\psi_Q: j\geq1,\ Q\in\Dcal_{j+J}\}
\]
with
\[
 f=\sum_{Q\in\Dcal_J}\ip{f}{\psi_Q^{(0)}}\psi_Q^{(0)}
 +\sum_{j\geq1}\sum_{Q\in\Dcal_{j+J}}\ip{f}{\psi_Q}\psi_Q
\]
and the exact $L^2$ Parseval identity. Replacing sums over $j\in\Z$ in Definitions~\ref{def:transform-spaces} and \eqref{eq:b-sequence}--\eqref{eq:f-sequence} by the low-frequency block plus $j\geq1$ yields the inhomogeneous spaces $B^s_{p,q}(dw)$ and $F^s_{p,q}(dw)$. The proof of Theorem~\ref{thm:frame-spaces} is unchanged and gives
\[
 s\in\R,\quad0<p,q\leq\infty
\]
for $B^s_{p,q}(dw)$ and
\[
 s\in\R,\quad0<p<\infty,\quad0<q\leq\infty
\]
for $F^s_{p,q}(dw)$. These definitions are again made directly through the Dunkl transform.

There is an analogous inhomogeneous atomic characterization.  For the
blocks at scales $j\geq1$ one uses exactly the atoms of
Definition~\ref{def:dunkl-atom}; for the single low-frequency block one uses
compactly supported atoms with the same size and Dunkl-derivative bounds but
without moment conditions.  The preceding atomic synthesis argument, with the low-frequency block kept separate, gives the atomic-to-transform estimate.  For the reverse estimate, the annular wavelets are decomposed by Proposition~\ref{prop:Latom-to-atoms}; the low-frequency frame elements are treated by the same orbit-shell decomposition with the moment-correction terms omitted.  Their kernel has arbitrary orbit decay, which gives a summable geometric tail.  Thus the resulting inhomogeneous atomic spaces coincide with $B^s_{p,q}(dw)$ and $F^s_{p,q}(dw)$ under the same sufficient conditions on $M$ and $L_0$.

\section*{Acknowledgements}
This work was supported by the Australian Research Council under grant DP260101083.
Chat GPT 5.6 was used during the preparation of this manuscript to assist with language editing and the checking of some technical estimates. All mathematical ideas, arguments, and final responsibility for the results are those of the author.

\end{document}